\documentclass[11pt,letterpaper]{amsart}

\usepackage{mathrsfs}

\usepackage{amssymb}

\usepackage{amscd}

\usepackage{tikz-cd}
\usepackage{todonotes}

\usepackage[all]{xy}

\usepackage{mathrsfs}
\usepackage{amsfonts}
\usepackage[backref=page,colorlinks,linkcolor=blue,citecolor=blue, pdfstartview=FitH]{hyperref}

\numberwithin{equation}{section}

\theoremstyle{plain}
\newtheorem{thm}{thm}[section]
\newtheorem{theorem}[thm]{Theorem}

\newtheorem{corollary}[thm]{Corollary}
\newtheorem{proposition}[thm]{Proposition}

\theoremstyle{definition}

\newtheorem{remark}[thm]{Remark}

\newtheorem{definition}[thm]{Definition}

\newtheorem{defn-thm}[thm]{Definition-Theorem}

\newcommand{\dbar}{\overline{\partial}}
\newcommand{\pa}{\partial}

\newcommand{\xu}{\sqrt{-1}}

\usepackage{fancyhdr}
\renewcommand{\leftmark}{{\scriptsize On the Kodaira dimension of some algebraic fiber spaces}}

\begin{document}
\title{On the Kodaira dimension of some algebraic fiber spaces}
\makeatletter
\let\uppercasenonmath\@gobble
\let\MakeUppercase\relax
\let\scshape\relax
\makeatother

\author{Yongpan Zou}

\address{Yongpan Zou, School of Sciences, Institute for Theoretic Sciences, Westlake University, Hangzhou, Zhejiang Province, 310030, China}

\email{598000204@qq.com}
\keywords{Pseudoeffective divisor, Kodaira dimension, Canonical bundle formula}

\footnote{
2020 \textit{Mathematics Subject Classification}.
Primary 32J25; Secondary 14E15.
}

\begin{abstract}
In this paper, we study the descent of positivity of the canonical bundle along fiber spaces. 
As a consequence, we prove a conjecture of Schnell, establishing the equivalence between the Non-vanishing Conjecture and its generalized version proposed by Campana and Peternell.
\end{abstract}

\maketitle
 \tableofcontents

\section{Introduction}

In the theory of algebraic fiber spaces, the canonical bundle formula stands as a foundational result. Initially introduced by Kodaira for elliptic surfaces, the formula was later generalized by Fujita and further developed systematically by Kawamata, Fujino--Mori, Ambro, Kollár, and others.
The canonical bundle formula is a far-reaching generalization of the adjunction formula. It decomposes the complex geometry of $X$
into the geometry of the base $Y$, the geometry of the generic fiber, and the combinatorial data of the singular fibers, providing an indispensable tool for the modern classification of algebraic varieties. In this paper, we use analytic methods in complex algebraic geometry,
including singular Hermitian metrics with semipositive curvature,
extensions of metrics, and related techniques.

The following is the general setup of Kawamata’s semi-positivity theorem, and we adopt the notation and ideas from \cite{Tak22}.

\textbf{Setup I}: Let $f: X\to Y$ be a surjective morphism of smooth projective varieties with connected fibers. Let $P=\sum P_i$ (respectively $Q=\sum Q_j$) be a reduced simple normal crossing divisor on $X$ (respectively, on $Y$) such that $f^{-1}(Q)\subset P$ and $f$, and assume that $f$ is smooth over $Y\setminus Q$. Decompose \( P \) into its horizontal and vertical parts with respect to \( f \) as $P = P^h + P^v$, where the horizontal part is given by \( P^h = \sum_{i,\, f(P_i) = Y} P_i \), and the vertical part by \( P^v = \sum_{i,\, f(P_i) \neq Y} P_i \). Then \( f: \operatorname{Supp}(P^h) \to Y \) is relatively normal crossing over \( Y \setminus Q \), and \( f(\operatorname{Supp}(P^v)) = Q \).

\begin{theorem}[\cite{Kaw98}, Theorem 2] \label{Kawa}
In addition to Setup I, let \( D = \sum d_i P_i \) be a \( \mathbb{Q} \)-divisor on \( X \) satisfying the following conditions:
\begin{enumerate}
  \item \textit{The pair \( (X, D) \) is sub-klt; that is, the positive part of pair \( (X, D_+ := \sum_{d_i > 0} d_i P_i) \) is klt, meaning \( d_i < 1 \) (note that \( d_i < 0 \) is allowed). In particular, the round-up \( \lceil -D \rceil \) is effective.}
  
  \item \textit{The natural map \( \mathcal{O}_Y \to f_* \mathcal{O}_X(\lceil -D \rceil) \) is generically surjective.}
  
  \item \textit{There exists a \( \mathbb{Q} \)-divisor \( L \) on \( Y \) such that \( K_X + D \sim_{\mathbb{Q}} f^*(K_Y + L) \).}
\end{enumerate}
Let
\begin{align*}
&f^* Q_j = \sum_i a_{ji} P_i \quad \text{with } a_{ji} \in \mathbb{Z}_{>0},\\
&\bar{d}_i = \frac{d_i + a_{ji} - 1}{a_{ji}} \quad (\in \mathbb{Q}_{<1}) \quad \text{if } f(P_i) = Q_j,\\
&\beta_j = \max \{ \bar{d}_i \mid f(P_i) = Q_j \} \quad (\in \mathbb{Q}_{<1}),\\
& B_D = \sum_j \beta_j Q_j.
\end{align*}

By construction, the pair \( (Y, B_D) \) is sub-klt. Let $M_Y = L - B_D$, we have
\[
K_X + D \sim_{\mathbb{Q}} f^*(K_Y + B_D + M_Y).
\]
Then the \( \mathbb{Q} \)-divisor \( M_Y \) is nef.
\end{theorem}


Takayama, in \cite{Tak22}, strengthens Kawamata’s theorem by showing that the moduli part $M_Y$ in the canonical bundle formula admits a natural singular Hermitian metric with semi-positive curvature and vanishing Lelong numbers (logarithmic singularities). His main tools are the existence of Bergman kernel type singular Hermitian metrics on pluri\mbox{-}adjoint bundles, as introduced by Berndtsson, P\u{a}un, and Takayama~\cite{BP08, PT18}, together with his results on the singularities arising from fiber integration, which also play an essential role in our paper. It is worth noting that Bakker, Filipazzi, Mauri, and Tsimerman \cite{BFMT25} recently proved the b-semiampleness of \(M_{Y'}\) for any Ambro model (not $M_Y$). However, in this paper, the canonical metric of $M_Y$ constructed in \cite{Tak22} plays the central role.
The following result is our main theorem.

\begin{theorem} [= Theorem \ref{claim}] \label{claim-0}
 In addition to Setup I, and assume 
  $$\mathrm{Codim}_Y \, f(\mathrm{Supp}\, D_+) \geq 2.$$
  If the canonical divisor \( K_X \) is pseudoeffective, then there exists an effective divisor $C_D=\sum_j \gamma_j  Q_j$ which is supported on $Q$ such that the divisor \( K_Y + M_Y + B_D + C_D \) on \( Y \) is pseudoeffective too.
\end{theorem}

To establish the positivity of $K_Y + M_Y +B_D+C_D$, we first analyze it over the smooth locus $Y_0$, where the $(B_D+C_D)$-component vanishes. Over $Y_0$, the positivity of $K_X$ fails to descend to $K_Y$ because $f$ is not locally trivial. However, introducing the divisor $M_Y$--which encodes the information of the variation of complex structures along the fibers--resolves this issue, yielding a positive metric for $K_Y + M_Y$ on $Y_0$. To extend this metric globally to $Y$, we must incorporate $(B_D+C_D)$ equipped with its canonical metric, as $(B_D+C_D)$ captures the singularities of both the fibers and the morphism $f$. We also expect that the assumptions in the theorem can be weakened for further applications.

The main (albeit modest) contribution of this paper is to clarify the role of the moduli divisor in the descent of positivity for algebraic fiber spaces. The key point is that the semipositive metric on the moduli divisor naturally contains a metric induced by the relative canonical bundle $K_{X/Y}$. Viewing the problem from this perspective allows us to address some questions in complex algebraic geometry. 

It is also worth noting that there has been extensive work on the descent of
positivity of the anti-canonical bundle along fiber spaces. For example, for a
smooth fibration \( f \colon X \to Y \), if \( -K_X \) is ample, nef, or
pseudoeffective, then the same property holds for \( -K_Y \). Such questions can
be naturally incorporated into the study of the positivity of the direct image
sheaf
\(
f_*\bigl((K_{X/Y}+L)\otimes \mathcal{I}\bigr),
\)
where \( \mathcal{I} \) denotes the multiplier ideal sheaf; see
\cite{BP08, PT18} for more details. Note that when \( L = -K_X \), this direct
image sheaf is precisely \( -K_Y \).
In contrast, the present paper focuses on the positivity of canonical bundle rather than the
anti-canonical bundle. As a result, new methods are required, although the
positivity of direct image sheaves continues to play an important role.

As the first application, we prove the following conjecture of C. Schnell. In \cite{Sch22}, Schnell conjectured that the Campana--Peternell conjecture and the non-vanishing conjecture are equivalent by proposing the following conjecture. 

\begin{theorem} [=Theorem \ref{schnell-2}] \label{schnell-3}
Let \( f: X \to Y \) be an algebraic fiber space between smooth projective varieties, and suppose that the general fiber \( F \) satisfies \( \kappa(F) = 0 \). If there exists an ample divisor \( H \) on \( Y \) such that \( m_0K_X - f^*H \) is pseudo-effective for some \( m_0 \geq 1 \), then \( \kappa(X) = \dim Y \).
\end{theorem}

Recall that the non-vanishing conjecture states that if $K_X$ is pseudo-effective, then there exists a sufficiently large and divisible multiple of $K_X$ that is effective. In \cite{CP11}, Campana and Peternell proposed the following generalization of the non-vanishing conjecture: Let $X$ be a smooth projective variety and $D$ be an effective divisor on $X$. If there exists an $m\geq 1$ such that the divisor class $mK_X-D$ is pseudo-effective. Then we should have $\kappa(X) \geq \kappa(X, D)$. The Campana--Peternell conjecture includes the non-vanishing conjecture as a special case when $D = 0$. On the other hand, assuming the Abundance conjecture, one can deduce the Campana--Peternell conjecture as a consequence. On the other hand, the Campana–Peternell conjecture has implications for various Iitaka-type conjectures; see \cite{Sch22, PS23} for more details.

Schnell proved his conjecture under the assumption that \(K_Y\) is pseudoeffective in \cite{Sch22}; see also \cite{Kim24} for partial results.
The conjecture here is slightly different from the original formulation in \cite[Conjecture~10.1]{Sch22}.
As \cite{Kim24} said, the advantage of this formulation is that the canonical bundle formula of Fujino--Mori arises naturally. 
Following \cite{FM00, Kol07}, we may write the canonical divisor \( K_X \) as:
$$ K_X + D \sim_{\mathbb{Q}} f^*(K_Y + B_D + M_Y),$$
where the decomposition satisfies the following properties:
\begin{enumerate} \label{three-cond-1}
\item \( f_* \mathscr{O}_X (\lfloor iD_{-} \rfloor) = \mathscr{O}_Y \) for all \( i > 0 \); 
\item \( \operatorname{codim}_Y f(\operatorname{Supp} D_{+}) \geq 2 \);
\item \( H^0(X, mK_X) = H^0(Y, m(K_Y + B_D + M_Y)) \) for all sufficiently divisible integers \( m \) (by canonical bundle formula and equality (1)).
\end{enumerate}

In this framework, the proof of Schnell's conjecture~\ref{schnell-3} reduces to establishing the bigness of the divisor \(K_Y + B_D + M_Y\). Hence, we should consider the problem of the descent of positivity along the fiber space. Since our assumptions are stronger than those in Theorem~\ref{claim-0}, we may take \(C_D = 0\).


We believe that our methods can also be applied to the study of the Iitaka conjecture. As an illustration, we provide new direct proofs of the following classical results in algebraic geometry related to the Iitaka conjecture, using the analytically natural methods developed in this paper. For related developments, the reader may consult \cite{PS23, Par23}. 

\begin{corollary}[=Theorem \ref{G}] \cite[Proposition G]{PS23}
Let \(f \colon X \to Y\) be a smooth algebraic fiber space with connected fibers, and denote by \(F\) a general fiber. 
Assume that \(\kappa(X) \geq 0\).
Then $ K_Y$  is pseudoeffective.
\end{corollary}

\begin{corollary} [= Theorem \ref{A}] \cite[Theorem A]{PS23} 
Let \(f \colon X \to Y\) be a smooth algebraic fiber space with connected fibers, and denote by \(F\) a general fiber. 
Assume that \(\kappa(F) \geq 0\).
Then
\[
Y \text{ is of general type }
\quad \Longleftrightarrow \quad
\kappa(X) = \kappa(F) + \dim Y .
\]
\end{corollary}

We also explain the following classic result of Viehweg--Zuo from our perspective.

\begin{corollary}[=Theorem \ref{VZ}] \cite[Theorem 0.2]{VZ01}
Let \(X\) be a complex projective manifold of non\text{-}negative Kodaira dimension. Then any surjective morphism \(f: X \to \mathbb{P}^1\) has at least three singular fibres.
\end{corollary}

In the last part of this paper, we will show how the Ohsawa--Takegoshi theorem can be used to study the existence of sections of the canonical bundle in an algebraic fiber space.

\vspace{1em}

 \noindent\textbf{Acknowledgement}: The author would like to thank Professors Shigeharu Takayama, Shin-ichi Matsumura, and Masataka Iwai for many helpful discussions.
The author is partially supported by National Key R\&D Program of China 2024YFA1014800 and NSFC No.12401073.


\section{Preliminaries} \label{first-section}
In this section, we first introduce some basic notations and definitions in complex geometry. The standard reference is \cite{Dem12a, Dem12b}.

\subsection{Positive metrics} 
Let $X$ be a compact K\"ahler manifold, and $L$ be a line bundle on $X$. We study $L$ from the viewpoint of metrics and curvatures.
\begin{definition}
\begin{enumerate}
        \item The line bundle $L$ is positive in the sense of curvature if it carries a Hermitian metric $h_L$ such that its curvature form $\xu \Theta_{h_L}(L):= -\xu \pa \dbar \log h$ is a K\"ahler form, i.e., positive definite real $(1,1)$-form.
        \item The line bundle $L$ is semi-positive if it carries a Hermitian metric $h_L$ such that its curvature form $\xu \Theta_{h_L}(L)$ is semi-positive definite $(1,1)$-form.
\end{enumerate}
\end{definition}

Due to the famous Kodaira embedding theorem, positivity $(1)$ and ampleness are equivalent for smooth projective varieties. When comparing the positivity of two line bundles $A$ and $B$, we write $ A\geq B$ if $A-B:=A\otimes B^{-1}$ is a semi-positive line bundle. Although the metrics mentioned are smooth, many applications require us to address cases where the metric of the line bundle is singular.

\begin{definition}
Let $(F,h)$ be a holomorphic line bundle on a complex manifold $X$ endowed with a possible singular Hermitian metric $h$. For any given trivialization $\theta: F|_{\Omega} \simeq \Omega \times \mathbb{C}$ by
$$ \| \xi \|_h = |\theta(\xi)| e^{-\phi(x)},  \quad x \in \Omega, \xi \in F_x,
$$
where $\phi \in L^1_{loc}(\Omega)$ is a function, called the weight of the metric. The curvature $\sqrt{-1} \Theta_h(F)$ of $h$ is defined by
$$  \sqrt{-1} \Theta_h(F) = \sqrt{-1} \partial \overline{\partial} \phi.
$$
The Levi form $\sqrt{-1}\partial\bar{\partial} \phi$ is understood in the sense of distributions, meaning that the curvature is a $(1,1)$-current rather than necessarily being a smooth $(1,1)$-form. This curvature is globally defined on $X$ and remains independent of the choice of trivializations. 
\end{definition}

\begin{proposition} \label{pseff metric}
   When \(X\) is a smooth projective variety and \(L\) is a Cartier divisor (or line bundle), the following two definitions of pseudoeffectiveness and bigness are equivalent:
\begin{itemize}
    \item \(L\) is big if its Iitaka dimension satisfies \(\kappa(X,L) = \dim X\).
    \item \(L\) is pseudo-effective if it lies in the closure of the cone of effective divisors in the Néron--Severi space \(N^1(X)\). In other words, \(L\) is pseudo-effective if it can be approximated arbitrarily closely by effective divisors.
    \item \(L\) is big if there exists a singular Hermitian metric \(h\) on \(L\) such that its curvature satisfies \(\Theta_{(L,h)} \geq \varepsilon \omega\) in the sense of currents, for some \(\varepsilon > 0\).
    \item \(L\) is pseudo-effective if there exists a singular Hermitian metric \(h\) on \(L\) such that \(\Theta_{(L,h)} \geq 0\) in the sense of currents.
\end{itemize}
\end{proposition}

A line bundle is usually called pseudoeffective if it admits a Hermitian metric whose weight function is plurisubharmonic.

\begin{definition} 
A function $u: \Omega \rightarrow [-\infty, \infty)$ defined on a open subset $\Omega \in \mathbb{C}^n$ is called plurisubharmonic (\text{Psh}, for short) if
\begin{enumerate}
\item $u$ is upper semi-continuous, this means that for every point \( z \in \Omega \),
\[
u(z) \geq \limsup_{w \to z} u(w).
\]
\item for every complex line $Q \subset \mathbb{C}^n$, $u|_{\Omega \cap Q}$ is subharmonic on $\Omega \cap Q$. That is, $u$ satisfies the mean value inequality. The latter property can be reformulated as follows: for all $a \in \Omega$, $\xi \in \mathbb{C}^n$ with $|\xi| = 1$, and $r > 0$ such that $\overline{B}(a, r) \subset \Omega$,
\begin{equation*}
u(a) \leq \frac{1}{2\pi} \int_0^{2\pi} u(a + re^{i\theta} \xi) \, d\theta. 
\end{equation*}
\end{enumerate}

A quasi-plurisubharmonic (quasi-psh, for short) function is a function $v$  that can be locally expressed as the sum of a Psh function and a smooth function. Let $\phi$ be a quasi-psh function on a complex manifold $X$, its multiplier ideal sheaf $\mathcal{I}(\phi) \subset \mathcal{O}_X$ is defined by
$$ \Gamma(U, \mathcal{I}(\phi)) = \{f \in \mathcal{O}_X(U): |f|^2e^{-\phi} \in L^1_{loc}(U) \}
$$
for every open set $U \subset X$. For a line bundle $(F, h)$, if the local weight of metric $h$ is $\phi$, we denote the multiplier ideal sheaf interchangeably by $\mathcal{I}(\phi)$ or $\mathcal{I}(h)$.
\end{definition}

For example, for any holomorphic functions $g, g_i$ on $\Omega$, both $\log|g|^{2\alpha}$ and $\log(\sum|g_i|^{2\alpha_i})$ are psh functions on $\Omega$. The multiplier ideal sheaf $\mathcal{I}(\log|g|^{2\alpha})= \mathcal{O}_{\Omega}$ if the positive number $\alpha$ sufficiently small.
The concept of the Lelong number described below characterizes the singularities of functions or, more generally, of currents.

\subsection{Extend sigular metric}
For our purpose in the following sections, we now introduce positivity notions for singular Hermitian metrics for vector bundles (in this paper, we only use the line bundle case).
Let $H_r$ be the space of semi-positive, possibly unbounded Hermitian forms on $\mathbb{C}^r$. A singular Hermitian metric $h$ on a vector bundle $E$ is a measurable map from $X$ to $H_r$ such that $h(x)$ is finite and positive definite almost everywhere. In particular, we require $0< \det h < +\infty$ almost everywhere.

At first, let \( u \) be an arbitrary holomorphic section of \( E \) with smooth Hermitian metric $h$. Then a short computation yields
\begin{equation*}
\xu \partial\bar{\partial}\|u\|_{h}^{2} = -\langle \xu\Theta u,u\rangle_{h} + \xu\langle D^{\prime}u,D^{\prime}u\rangle_{h} \geq -\langle \xu\Theta u,u\rangle_{h}.
\end{equation*}
Hence, we see that if the curvature is negative in the sense of Griffiths, then \( \|u\|_{h}^{2} \) is plurisubharmonic. On the other hand, we can always find a holomorphic section \( u \) such that \( D^{\prime}u = 0 \) at a point. Thus \( h \) is negatively curved in the sense of Griffiths if and only if \( \|u\|_{h}^{2} \) is plurisubharmonic for any holomorphic section \( u \).

The Griffiths negativity of the metric $h$ can be characterized equivalently through plurisubharmonicity conditions on \( \log \|u\|_{h}^{2} \). To see this, observe that the property can be deduced from a fundamental fact about plurisubharmonic functions: for any positive-valued function $v$, the logarithm $\log v$ is plurisubharmonic if and only if $v e^{2\,\mathrm{Re}\,q}$ is plurisubharmonic for every polynomial $q$. Applying this to $v = \|u\|_{h}^2$, we note that
\(
\|u\|_{h}^2 e^{2\,\mathrm{Re}\,q} = \|u e^{q}\|_{h}^2,
\)
where $u e^{q}$ remains a holomorphic section since $q$ is a polynomial. By assumption, $\|u e^{q}\|_{h}^2$ is plurisubharmonic (as $h$ is negatively curved), which implies $\log |u\|_{h}^2$ satisfies the plurisubharmonicity condition.

From this perspective, one can generalize the concept of Griffiths positivive curvature to the setting of holomorphic vector bundles endowed with singular metrics.
\begin{definition} \cite{BP08, PT18}\label{def Grif semi-posi as sing}
Let $(E,h)$ be a singular Hermitian metric on $X$, then $(E,h)$ is said to be:
\begin{enumerate}
		\item \textit{semi-negative curved} if $ \|s\|^2_h$ (or $\log \|s\|^2_h$) is psh for any local holomorphic section $s$ of $E$.
		\item \textit{semi-positive curved} if the dual metric $h^\star$ on $E^\star$ is semi-negative curved.
	\end{enumerate}
\end{definition}

If $E$ is a line bundle with a singular Hermitian metric $h$, the semi-positive curved condition in the above definition and pseudo-effective condition in Proposition \ref{pseff metric} coincide.

In the proof of the main theorem, we first construct a metric on a certain Zariski open subset. The following proposition provides a criterion for extending this metric from the Zariski open subset to the entire space.

 \begin{proposition}  \cite[Proposition $2.4$]{Tak24} \label{Tak24}
Let \( X_0 \subset X \) be a non-empty analytic Zariski open subset of the complex manifold $X$. Suppose \( E{|_{X_0}} \) admits a singular Hermitian metric \( h^\circ \) with Griffiths semi-negative curvature. Suppose further that for any open subset \( U \subset X \) and any \( s \in H^0(U, E) \), the psh function \(\log \|s\|_{h^\circ}^2\) on \( U \cap X_0 \) extends to a psh function on \( U \). (This extension condition is automatically satisfied if \(\operatorname{codim}(X \setminus X_0) \geq 2\), by virtue of the Hartogs type extension of psh functions.) Then \( h^\circ \) extends uniquely to a singular Hermitian metric on \( E \) with Griffiths semi-negative curvature.
\end{proposition}

\section{Main theorem}
In this section, we state the main theorem and defer its proof to the next section. We then recall some facts on the complex structure and positivity of the moduli divisor.

\subsection{Setup II}
 Under the assumptions in Setup I from the introduction, let \( D = D_+ - D_- \), where the positive part is given by \( D_+ = \sum_{d_i > 0} d_i P_i \) and the negative part by \( D_- = \sum_{d_i < 0} (-d_i) P_i \), with no common components. Moreover, we assume that
\begin{equation*}
    \mathrm{Codim}_Y \, f(\mathrm{Supp}\, D_+) \geq 2.
\end{equation*}
In particular, the contribution from \(D_+ \) can be ignored for our purposes.

We further decompose the negative part \( D_- \) into its \( f \)-horizontal and \( f \)-vertical components:
\[
D_- = D_-^h + D_-^v.
\]
Without loss of generality, we rewrite the canonical bundle formula as
\begin{equation} \label{first}
    K_X + D = K_X - D_-^h - D_-^v \sim_{\mathbb{Q}} f^*(K_Y + M_Y + B_D).
\end{equation}
This relation holds on a Zariski open subset of \( X \) whose complement has codimension at least two. In what follows, we analyze the vertical part \( D_-^v = \sum (-d_i) P_i \), where \( d_i < 0 \), in greater detail. Let
\begin{align*}
&f^* Q_j = \sum_i a_{ji} P_i \quad \text{with } a_{ji} \in \mathbb{Z}_{> 0},\\
&c_i = -\frac{d_i}{a_{ji}} \quad  \text{ if } f(P_i) = Q_j, \text{ here } d_{i} < 0,\\
&\gamma_j = \max \{ c_i | \ f(P_i) = Q_j \}.
\end{align*}
We set $C_D = \sum_j \gamma_j  Q_j$, note that
\begin{equation} \label{ej}
B_D + C_D = \sum e_j Q_j
\end{equation}
is an effective divisor.
Indeed, each coefficient \( e_j \) satisfies the inequality
$$
e_j=\beta_j+\gamma_j \geq \frac{d_i+a_{ji}-1}{a_{ji}} + \frac{-d_i}{a_{ji}}\geq \frac{a_{ji}-1}{a_{ji}}\geq 0.
$$
where \( d_i < 0 \) for the components of the vertical part \( D^v_{-} \).
We can also write the canonical bundle formula as:
\begin{equation} \label{second}
    K_X - D^h_{-} + f^*C_D - D^v_{-} \sim_{\mathbb{Q}} f^*(K_Y + M_Y + B_D + C_D).
\end{equation}
Note that $f^*C_D - D^v_{-}= f^*C_D + \sum d_i P_i$ is an effective $\mathbb{Q}$ divisor. We now state the main theorem of this section.
\begin{theorem} \label{claim}
  Assume \textbf{Setup II}. If the canonical divisor \( K_X \) is pseudo-effective, then so is the divisor \( K_Y + M_Y + B_D + C_D \) on \( Y \).
\end{theorem}

\subsection{Complex structure}
It is well known that the moduli divisor \( M_Y \) in \eqref{first} reflects the variation of complex structures among the fibers. To make this precise, we recall some basic facts about the deformation theory of complex structures.
In this subsection, let \( Z \) be a fixed compact complex manifold of dimension \( n \), and let \( \Delta \subset \mathbb{C}^m \) be a small open ball centered at the origin. A family of deformations of \( Z \) over \( \Delta \) is a proper, holomorphic submersion
\[
\pi: \mathcal{Z} \to \Delta,
\]
where \( \mathcal{Z} \) is a complex manifold and the central fiber is \( Z = \pi^{-1}(0) \). Each fiber
\[
\mathcal{Z}_t := \pi^{-1}(t), \quad t \in \Delta,
\]
is a compact complex manifold diffeomorphic to \( Z \), but typically equipped with a different complex structure. The reader should think of the fibers \( \mathcal{Z}_t \) as small deformations of the central fiber \( Z \), capturing how the complex structure varies in families.

\begin{remark}
    A smooth family $f \colon M \to B$ is called trivial if it is (biholomorphically) isomorphic to the product $M_b \times B \to B$ for some (and hence all) $b \in B$. It is called locally trivial if there exists an open covering $B = \bigcup U_\alpha$ such that every restriction $f \colon f^{-1}(U_\alpha) \to U_\alpha$ is trivial.

  \noindent The family is locally trivial if and only if a certain morphism $\mathcal{KS}$ of sheaves over $B$ is trivial, while the restriction of $\mathcal{KS}$ at a point $b \in B$ is a linear map
\[
\mathrm{KS} \colon T_{b}B \to H^1(M_b, T M_b),
\]
called the Kodaira-Spencer map, which can be interpreted as the first derivative at the point $b$ of the map
\[
B \to \{\text{isomorphism classes of complex manifolds}\}, \quad b \mapsto M_b.
\]
On the other hand, recall the Fischer--Grauert theorem states:  
Let \( f \colon M \to B \) be a holomorphic family of compact complex manifolds.  
Then all fibers of \( f \) are biholomorphic if and only if \( f \) is locally trivial.

\end{remark}

Back to the family $\pi$. By Ehresmann’s fibration theorem, we have \( \mathcal{Z} \simeq Z \times \Delta \) as differentiable manifolds (but not, in general, as complex manifolds). However, it can be shown that any family \( \mathcal{Z} \to \Delta \) admits a transversely holomorphic trivialization; (see \cite[Theorem 4.31]{Man04} for example):
\begin{equation} \label{holo trivi}
(p, \pi): \mathcal{Z} \to Z \times \Delta.
\end{equation}
In other words, one can always find a map \( p: \mathcal{Z} \to Z \) (typically not holomorphic) such that:
\begin{enumerate}
    \item \( (p, \pi) \) is a diffeomorphism;
    \item \( p \) restricts to the identity on the central fiber \( \pi^{-1}(0) = Z \); and
    \item the fibers of \( p \) are holomorphic submanifolds of \( \mathcal{Z} \), transverse to \( Z \).
\end{enumerate}
The reader will note that any fiber of \( p \) is biholomorphic to \( \Delta \) via the holomorphic map \( \pi \).

Let $Z$ be a compact manifold with tangent bundle $TZ$. Fix an almost complex structure $J \in \mathrm{End}(TZ)$, which induces the canonical decomposition
\[
TZ \otimes \mathbb{C} = TZ^{1,0} \oplus TZ^{0,1}
\]
into $(1,0)$ and $(0,1)$ subbundles, with corresponding projection maps $\pi_{1,0}$ and $\pi_{0,1}$.

Now consider a deformation $J'$ of $J$ (viewed as a nearby almost complex structure). When $J'$ is sufficiently close to $J$, the restriction of $\pi_{0,1}$ to $TZ'^{0,1}$ (the $(0,1)$-subbundle of $J'$) becomes an isomorphism, yielding a canonical composition:
\[
\begin{tikzcd}
TZ^{0,1} \arrow[r, "\sim", "(\pi_{0,1}|)^{-1}"'] & TZ'^{0,1} \arrow[r, "\pi_{1,0}"] & TZ^{1,0}.
\end{tikzcd}
\]
This composite map naturally identifies with a section $\xi \in A^{0,1}(TZ^{1,0})$, i.e., a $TZ^{1,0}$-valued $(0,1)$-form on $Z$. 
Remarkably, this correspondence works both ways: every sufficiently small $\xi \in A^{0,1}(TZ^{1,0})$ determines a deformation $J'$ by specifying its $(0,1)$-subbundle as the graph of $\xi$ in $TZ \otimes \mathbb{C}$.

When expressed in local coordinates, this becomes particularly concrete. Let us assume the fixed almost complex structure $J$ is integrable, with local holomorphic coordinates $(z^1,\ldots,z^n)$. In these coordinates, the deformation tensor $\xi \in A^{0,1}(TZ^{1,0})$ takes the explicit form:

\[
\xi = \sum_{i,u} h^u_i(z) \, d\bar{z}^i \otimes \frac{\partial}{\partial z^u},
\]
where each coefficient $h^u_i(z)$ is a smooth function of the coordinates. 
The deformed subbundle $TZ'^{0,1}$ is locally generated by the frame:
\[
\left\{
\frac{\partial}{\partial \bar{z}^i} + \sum_{u=1}^n h^u_i(z) \frac{\partial}{\partial z^u} \right\}_{i=1}^n,
\]
which visibly represents a tilting of the original anti-holomorphic directions by holomorphic vector field components.

Next, we analyze the family case $\mathcal{Z} \simeq Z \times \Delta$, enabled by the transversely holomorphic trivialization. For computational convenience, we formulate the setup in local coordinates. Let \( z^1, \ldots, z^n \) be local holomorphic coordinates on \( Z \), and let \( t^1, \ldots, t^m \) be holomorphic coordinates on the disk \( \Delta \). Of course, \( z^1, \ldots, z^n, t^1, \ldots, t^m \) are not generally holomorphic coordinates for the deformed complex structure \( J' \); however, together with their complex conjugates, they provide a smooth coordinate system that can be used to describe vector fields and differential forms on \( Z \times \Delta \).

The deformation \( \xi(t) \in A^{0,1}(TZ^{1,0}) \) may then be written in the form
\[
\xi(t) = \sum_{i, u} h_i^u(z, t) \, d\overline{z}^i \otimes \frac{\partial}{\partial z^u},
\]
where \( h_i^u(z, 0) = 0 \). At each point \( (z, t) \in Z \times \Delta \), the corresponding \( (0,1) \)-subspace for \( J' \) is spanned by the \( m + n \) vector fields
\[
\frac{\partial}{\partial \overline{z}^i} + \sum_u h_i^u(z, t) \frac{\partial}{\partial z^u} \quad \text{and} \quad \frac{\partial}{\partial \overline{t}^k} \qquad (1 \leq i \leq n,\; 1 \leq k \leq m).
\]
The one-forms that vanish on this basis are precisely
\begin{equation} \label{complex str}
dz^u - \sum_i h_i^u(z,t) \, d\overline{z}^i \quad \text{and} \quad dt^k \qquad (1 \leq u \leq n,\; 1 \leq k \leq m).
\end{equation}

In summary, we have a transversely holomorphic trivialization such that the pullback of the local holomorphic $1$-form $dz^u$ on $Z \times \Delta$ to $\mathcal{Z}$ is given by
\[
dz^u - \sum_i h_i^u(z,t) \, d\overline{z}^i,
\]
while the forms $dt^k$ remain unchanged.

\subsection{Moduli divisor in canonical bundle formula}
Next, we require a semi-positive metric on the moduli part \( M_Y \) in~\eqref{first} that possesses suitable analytic properties. This metric was constructed by Takayama in \cite{Tak22}. For completeness, we briefly recall the construction below. By replacing the divisor \( D \) with \( D - f^*B_D \), we may assume for the moment that \( B_D = 0 \). This simplification allows us to focus on the behavior of the moduli divisor \( M \) without the contribution from the boundary part.

Recall that under the conditions of Setup~II, we need only consider the negative part of $D$. We therefore assume $D = -D_{-}$ and define
\[
\Delta_D := \lceil D_{-} \rceil - D_{-}.
\]
where $\lceil \cdot\rceil$ denotes the round-up of $\mathbb{Q}$-divisors.
 We have
\begin{equation*}
K_{X/Y} + \Delta_D - \lceil D_{-} \rceil \sim_{\mathbb{Q}} f^*M_Y
\end{equation*}
in our setup. Define
\(
N := \lceil D_{-} \rceil - K_{X/Y},
\)
which is a line bundle on \( X \). Then we have the relations
\[
K_{X/Y} + N = \lceil D_{-} \rceil, \quad \text{and} \quad \Delta_D - N \sim_{\mathbb{Q}} f^*M_Y
\]
for the \( \mathbb{Q} \)-divisor \( M_Y \) on \( Y \).
We now equip the bundles \( K_{X/Y} \), \( N \), and \( \Delta_D \) with smooth Hermitian metrics \( h_K \), \( h_N \), and \( h_{\Delta_D} \), respectively. Then we define the closed \( (1,1) \)-form
\[
\theta := \frac{\sqrt{-1}}{2\pi} \left( \Theta_{h_K} + \Theta_{h_{\Delta_D}} \right) \in c_1(K_{X/Y} + \Delta_D) = c_1(K_{X/Y} + D + \lceil D_{-} \rceil).
\]

Associated to the effective \( \mathbb{Q} \)-divisor \( \Delta_D \), we define a canonical singular Hermitian metric
\begin{equation} \label{root}
g_\delta := \frac{1}{|s_{\Delta_D}|^2} = \frac{h_{\Delta_D}}{|s_{\Delta_D}|_{h_{\Delta_D}}^2},
\end{equation}
where \( s_{\Delta_D} \) is a section of \( \Delta_D \), and \( h_{\Delta_D} \) is a smooth Hermitian metric on the line bundle \( \mathcal{O}_X(\Delta_D) \). To clarify the construction of the metric \( g_\delta \), we provide some details. Choose a positive integer \( k \) such that \( k\Delta_D \) is an integer effective divisor. Then there exists a canonical global section \( s \in H^0(X, \mathcal{O}_X(k\Delta_D)) \) defining \( k\Delta_D \), and we may equip the corresponding line bundle \( \mathcal{O}_X(k\Delta_D) \) with a smooth Hermitian metric \( h \).
Although the \( k \)-th root of the section \( s \) is not globally well-defined, the \( k \)-th root of the metric \( h, |s|^2_h \) is well-defined. Thus, we define
\[
h_{\Delta_D} := h^{\frac{1}{k}}, \qquad \text{and} \qquad |s_{\Delta_D}|_{h_{\Delta_D}}^2 := |s|_h^{\frac{2}{k}}.
\]
The curvature current of \( g_\delta \) is given by
\[
\frac{\sqrt{-1}}{2\pi} \Theta_{g_\delta} = [\Delta_D] \geq 0,
\]
where \( [\Delta_D] \) denotes the current of integration along the divisor \( \Delta_D \).
Since the pair \( (X, \Delta_D) \) is klt, the associated multiplier ideal sheaf is trivial:
\(
\mathcal{I}(g_\delta) = \mathcal{O}_X.
\)
Moreover, for every fiber \( X_y := f^{-1}(y) \) with \( y \in Y \setminus Q \), the restriction satisfies
$\mathcal{I}(g_\delta|_{X_y}) = \mathcal{O}_{X_y}$.

Let \( \{Y_i\}_{i \in I} \) be a locally finite open cover of \( Y \) such that each \( Y_i \) is sufficiently small and biholomorphic to a ball in \( \mathbb{C}^m \), where \( m = \dim Y \). Denote \( X_i := f^{-1}(Y_i) \). We may assume that \( M_Y \sim_{\mathbb{Q}} 0 \) on each \( Y_i \), so that on every \( X_i \), we have
\[
N_i := N|_{X_i} \sim_{\mathbb{Q}} \Delta_D|_{X_i}, \quad \text{and hence} \quad K_{X/Y} + N_i \sim_{\mathbb{Q}} K_{X/Y} + \Delta_D|_{X_i}.
\]
A singular Hermitian metric on \( N_i \) (respectively, on \( K_{X/Y} + N_i \)) induces a singular Hermitian metric on the \( \mathbb{Q} \)-line bundle \( \Delta_D|_{X_i} \) (respectively, \( K_{X/Y} + \Delta_D|_{X_i} \)), and vice versa. Via the isomorphism \( N_i \sim_{\mathbb{Q}} \Delta_D|_{X_i} \), we define a smooth Hermitian metric
\begin{equation} \label{hi}
h_i := h_K \cdot h_{\Delta_D}|_{X_i}
\end{equation}
on the line bundle \( K_{X/Y} + N_i \), whose curvature form satisfies
\[
\frac{\sqrt{-1}}{2\pi} \Theta_{h_i} = \theta \quad \text{on } X_i.
\]
We also equip \( N_i \sim_{\mathbb{Q}} \Delta_D|_{X_i} \) with a singular Hermitian metric
\[
g_i := g_\delta|_{X_i}.
\]
This metric satisfies $\frac{\sqrt{-1}}{2\pi} \Theta_{g_i} = [\Delta_D|_{X_i}] \geq 0$,
and its multiplier ideal sheaf is trivial: \( \mathcal{I}(g_i) = \mathcal{O}_{X_i} \), as well as \( \mathcal{I}(g_i|_{X_y}) = \mathcal{O}_{X_y} \) for every \( y \in Y_i \setminus Q \).

We can see that \(\{h_i/g_i\}_{i \in I}\) gives a global singular Hermitian metric on \(K_{X/Y}\), i.e., \(h_i/g_i = h_j/g_j\) on \(X_i \cap X_j\). In fact, on each \(X_i\),
\[
\frac{h_i}{g_i} = h_K h_{\Delta_D} \left( \frac{h_{\Delta_D}}{|s_{\Delta_D}|^2_{h_{\Delta_D}}} \right)^{-1} = h_K |s_{\Delta_D}|^2_{h_{\Delta_D}}.
\]

We take a section \(\sigma = s_{\lceil D_{-}\rceil} \in H^0(X, \mathcal{O}_X(\lceil D_{-}\rceil))\) whose divisor is \(\lceil D_{-}\rceil\). As \(\mathcal{I}(g_i|_{X_y}) = \mathcal{O}_{X_y}\) for \(y \in Y \setminus Q\), by Ohsawa-Takegoshi type \(L^2\)-extension and the base change theorem, for every \(y \in Y \setminus Q\), we have \(h^0(X_y, K_{X_y} + N_i) = h^0(X_y, \lceil D_{-}\rceil) = 1\). We then can see \(H^0(X_y, K_{X/Y} + N_i) = \mathbb{C}\sigma_y\) with \(\sigma_y := s_{\lceil D_{-}\rceil}|_{X_y}\) for \(y \in Y_i \setminus Q\). Hence, for every \(y \in Y_i \setminus Q\), there exists \(\tau_y \in H^0(X_y, \lceil D_{-}\rceil)\) such that
\[
\int_{X_y} g_i c_n \tau_y \wedge \bar{\tau}_y = 1.
\]
Such \(\tau_y\) is unique up to a multiplication of a complex number of unit norm. In fact, letting
\begin{equation} \label{Fi}
F_i(y) = \int_{X_y} g_i c_n \sigma_y \wedge \bar{\sigma}_y \in (0, +\infty) \text{ and } \|\sigma_y\|_{g_i} = \sqrt{F_i(y)},
\end{equation}
we can take
\[
\tau_y := \frac{\sigma_y}{\|\sigma_y\|_{g_i}} \in H^0(X_i, K_{X/Y} + N_i),
\]
for \(y \in Y \setminus Q\). We see that \(F_i(y)\) and hence \(\|\sigma_y\|_{g_i}\) are \(C^\infty\)-smooth on \(Y \setminus Q\) by the fractional relative normal crossing property \(\Delta_D\) and the derivative under integration. We let \(\tau_i = \{\tau_y\}_{y \in Y \setminus Q} \in C^\infty(X_i \setminus f^{-1}(Q), K_{X/Y} + N_i)\) be a family of fiberwise holomorphic sections parametrized by \(Y_i \setminus Q\).

As shown by Takayama in \cite[Section $4$]{Tak22}, one can construct a Bergman kernel type metric $h_K \cdot h_{\Delta_D} \cdot e^{-\varphi}$ on \( K_{X/Y} + \Delta_D \). On each $X_i$, we have $\varphi= \varphi_i =\log |\tau_i|^2_{h_i}$. And its curvature current satisfies that \(\theta + dd^c \varphi  \geq  \lceil D_{-} \rceil\). Motivated by this, we consider a singular Hermitian metric \( \widetilde{h} \) on
\begin{equation} \label{moduli line}
K_{X/Y} + \Delta_D - \lceil D_{-} \rceil = K_{X/Y} + D \sim_{\mathbb{Q}} f^*M_Y,
\end{equation}
defined by
\begin{equation} \label{moduli metric}
\widetilde{h} := h_K \cdot h_{\Delta_D} \cdot e^{-\varphi} \cdot \frac{|s_{\lceil D_{-} \rceil}|^2_{h_K h_{\Delta_D}}}{h_K \cdot h_{\Delta_D}} = h_K \cdot h_{\Delta_D} \cdot e^{-\varphi} \cdot |s_{\lceil D_{-} \rceil}|^2,
\end{equation}
where \( \varphi \) is a globally defined function on \( X \). We note that a relation
\[
\frac{e^{\varphi_i}}{|s_{\lceil D_{-}\rceil}|^2_{h_i}}(x) = \frac{|s_{\lceil D_{-}\rceil}|^2_{h_i}}{F_i(f(x))} \cdot \frac{1}{|s_{\lceil D_{-}\rceil}|^2_{h_i}} = \frac{1}{F_i(f(x))}
\]
holds for any \( x \in X_i \setminus f^{-1}(Q) \).
Hence, on each local chart \( X_i \), combining \eqref{hi} we have $\widetilde{h}|_{X_i} = F_i$.

The curvature current is \(\Theta_{\widetilde{h}} = \theta + dd^c \varphi - \lceil D_{-} \rceil \geq 0\). Obviously, \(F_i\) is constant on each fiber (it can be \(\equiv +\infty\) on singular fibers), the induced metric on \(M_Y\) is \((f_* \widetilde{h})(y) = F_i(y) (= \|\sigma_y\|^2_{g_i})\) on each \(Y_i\). 

\section{The proof of the main theorem}
The proof proceeds in two parts: In the first part, we investigate the positivity properties of the metric of \( K_Y + M_Y + B_D + C_D \) over the smooth locus of \( f \). In the second part, we examine the asymptotic behavior of the metric near the boundary divisor \( Q \).
\subsection{Part $1$}

We begin by considering points \( y_0 \in Y_0:= Y \setminus Q \), where the morphism \( f \) is smooth. Let \( V \subset \mathbb{C}^m \) be a small open coordinate neighborhood (a polydisk or ball) centered at \( y_0 \), such that \( V \cap Q = \emptyset \). Set \( U := f^{-1}(V) \subset X \). 
Over this neighborhood, the morphism \( f: U\to V \) defines a smooth, proper submersion. By Ehresmann’s theorem, there exists a differentiable trivialization, and we may further assume the existence of a transversely holomorphic trivialization:
\[
U \simeq X_0 \times \Delta,
\]
where \( X_0 := f^{-1}(y_0) \) and \( \Delta \simeq V \) is a small ball in \( \mathbb{C}^m \). This identification allows us to treat \( U \) as a smooth family of complex manifolds parametrized by \( \Delta \), with varying complex structures in the transverse direction.

\begin{remark}
Let \( Z \) be a compact complex manifold of dimension \( n \). Its canonical bundle \( K_Z \) admits a local frame \( dz_1 \wedge \cdots \wedge dz_n \) on a coordinate chart \( \Omega_i \) with local holomorphic coordinates \( z_1, \ldots, z_n \). A holomorphic section of \( K_Z \) (if it exists) is given by a collection of holomorphic functions \( f_i \) defined on \( \Omega_i \) such that the forms \( f_i\, dz_1 \wedge \cdots \wedge dz_n \) glue together to define a global holomorphic \( n \)-form. From this perspective, we may regard
\[
c_n\, dz_1 \wedge \cdots \wedge dz_n \wedge d\overline{z}_1 \wedge \cdots \wedge d\overline{z}_n, \quad \text{where } c_n = (\sqrt{-1})^{n^2},
\]
as a Hermitian metric on \( K_Z \). On the other hand, given a holomorphic section \( f_i \), the expression \( \frac{1}{|f_i|^2} \) defines a singular Hermitian metric on \( K_Z \) with singularities along the zero locus of \( f_i \).

In general, any two Hermitian metrics \( h_1 \) and \( h_2 \) on the same holomorphic line bundle differ by a positive function, i.e.,
\[
h_1 = e^{-\varphi} h_2,
\]
where \( \varphi \) is a globally defined real-valued function.

Suppose the Hermitian metric on \( K_Z \) is locally expressed on a chart \( \Omega_i \subset Z \), with holomorphic coordinates \( z_1, \ldots, z_n \), as
\[
e^{-\varphi} \cdot c_n\, dz_1 \wedge \cdots \wedge dz_n \wedge d\overline{z}_1 \wedge \cdots \wedge d\overline{z}_n.
\]
We say that the local expression of the metric is \( e^{-\varphi} \).
Now, consider another overlapping chart \( \Omega_j \) with holomorphic coordinates \( w_1, \ldots, w_n \), such that \( \Omega_i \cap \Omega_j \neq \emptyset \). On the overlap, the change of coordinates gives
\[
dz_1 \wedge \cdots \wedge dz_n = \det\left( \frac{\partial z}{\partial w} \right) dw_1 \wedge \cdots \wedge dw_n.
\]
Therefore, the volume form (and hence the metric) transforms as
\[
e^{-\varphi} \cdot c_n\, dz_1 \wedge \cdots \wedge dz_n \wedge d\overline{z}_1 \wedge \cdots \wedge d\overline{z}_n 
= e^{-\varphi} \left| \det\left( \frac{\partial z}{\partial w} \right) \right|^2 \cdot c_n\, dw_1 \wedge \cdots \wedge dw_n \wedge d\overline{w}_1 \wedge \cdots \wedge d\overline{w}_n.
\]
Thus, the local expression of the metric in the \( w \)-coordinates becomes
\[
e^{-\varphi} \cdot \left| \det\left( \frac{\partial z}{\partial w} \right) \right|^2,
\]
and note that $ \det\left( \frac{\partial z}{\partial w} \right)$ is holomorphic.
\end{remark}

Returning to our setting, since \( K_X \) is assumed to be pseudoeffective, there exists a singular Hermitian metric \( h_1 = e^{-\psi_X} \) on \( K_X \), where \( \psi_X \) is a psh function. This means that the top-degree form
\begin{equation} \label{canonical metric}
\Psi_X := e^{\psi_X} \, dz_1 \wedge \cdots \wedge dz_n \wedge d\overline{z}_1 \wedge \cdots \wedge d\overline{z}_n \wedge dt_1 \wedge \cdots \wedge dt_m \wedge d\overline{t}_1 \wedge \cdots \wedge d\overline{t}_m
\end{equation}
is a globally defined \((n+m,n+m)\)-form on \( X \), where \( (z_1, \ldots, z_n; t_1, \ldots, t_m) \) are local coordinates in \( X  \). Since \( e^{\psi_X} \) is locally integrable due to the upper boundedness of psh functions, we can define the fiberwise integral
\[
\Psi_Y := \int_{X/Y} \Psi_X,
\]
which yields a globally defined \((m,m)\)-form on \( Y \). If we write this form locally as
\begin{equation*}
\Psi_Y = e^{\psi_Y} \, dt_1 \wedge \cdots \wedge dt_m \wedge d\overline{t}_1 \wedge \cdots \wedge d\overline{t}_m,
\end{equation*}
then we may interpret \( e^{\psi_Y} \) as defining the norm of a local section of \( -K_Y \) (or the metric of $-K_Y$). Our goal is to prove that the function $e^{\psi_Y}$ multiplied by certain metrics of $-M_Y$, $-B_D$, and $-C_D$ is psh. This would imply that: every local norm of any section of $-K_Y - M_Y - B_D - C_D$ is psh. Consequently, the bundle $K_Y + M_Y + B_D + C_D$ admits a Hermitian metric with semi-positive curvature in the sense of Definition~\ref{def Grif semi-posi as sing}.

\begin{remark}
For the reader’s convenience, we recall some basic facts concerning fiber integrals.
 We may regard the fiberwise integral as the pushforward of differential forms. 
Let $F \colon M_{1} \to M_{2}$ be a submersion between manifolds; that is, \(F\) is surjective and, for every \(x \in M_{1}\), the differential map 
\(
d_{x}F \colon T_{M_{1},x} \longrightarrow T_{M_{2},F(x)}
\)
is surjective. 
Let \(g\) be a differential form of degree \(q\) on \(M_{1}\), with \(L^{1}_{\mathrm{loc}}\) coefficients, such that the restriction 
$
F\big|_{\operatorname{Supp} g}
$
is proper. 
We claim that \(F_{*}g\) is the differential form of degree \(q - (m_{1} - m_{2})\) obtained from \(g\) by integrating along the fibers of \(F\), where \(m_{1} = \dim_{\mathbb{R}} M_{1}\) and \(m_{2} = \dim_{\mathbb{R}} M_{2}\). We denoted
\[
F_{*}g(y) = \int_{z \in F^{-1}(y)} g(z).
\]
This assertion is equivalent to the following generalized form of Fubini's theorem:
\[
\int_{M_{1}} g \wedge F^{*}u 
= \int_{y \in M_{2}} \left( \int_{z \in F^{-1}(y)} g(z) \right) \wedge u(y),
\quad \forall\, u \in {{}^{0}\mathscr{D}}^{m_{1}-q}(M_{2}).
\]
Using a partition of unity on \(M_{1}\) and the constant rank theorem, the verification of this formula is easily reduced to the case where 
\(M_{1} = A \times M_{2}\) and $F$ be the second projection. 
Let \(z = (x,y) \in A \times M_{2}\) 
be any point of \(M_{1}\), the above formula becomes
\[
\int_{A \times M_{2}} g(x,y) \wedge u(y) 
= \int_{y \in M_{2}} \left( \int_{x \in A} g(x,y) \right) \wedge u(y),
\]
where the direct image of \(g\) is computed from
\[
g = \sum_{|I| + |J| = q} g_{I,J}(x,y) \, dx_{I} \wedge dy_{J}
\]
by the formula
\begin{align*}
F_{\star}g(y) 
&= \int_{x \in A} g(x,y) \\
&= \sum_{|J| = q-r} \left( \int_{x \in A} g_{(1,\ldots,r),J}(x,y) 
\, dx_{1} \wedge \cdots \wedge dx_{r} \right) dy_{J}.
\end{align*}
In this situation, we see that \(F_{\star}g\) has \(L^{1}_{\mathrm{loc}}\) coefficients 
on \(M_{2}\) whenever \(g\) is \(L^{1}_{\mathrm{loc}}\) on \(M_{1}\).
\end{remark}

Back to our proof. The upper semi-continuity of \( e^{\psi_Y} \) follows by approximation: we take a decreasing sequence of smooth functions \( e^{\psi^i_X} \to e^{\psi_X} \), and use the smoothness of \( f \colon X \to Y \) over the smooth locus to pass to the limit under fiber integration. Recall that any upper semi-continuous function admits a decreasing approximation by smooth functions. Conversely, the limit of any decreasing sequence of continuous (or smooth) functions remains upper semi-continuous.

As we restrict to the open small ball \( V \subset Y \) where \( \mathrm{Supp}(B_D + C_D) \cap V = \emptyset \), it is enough to analyze the positivity properties of the bundle \( K_Y + M_Y \) over \( V \). Note that \( B_D + C_D \) is an effective \( \mathbb{Q} \)-divisor and is therefore pseudoeffective.

We now consider the fiber integral
\[
\Psi_{Y, M} := e^{\psi_{Y, M}} \, dt_1 \wedge \cdots \wedge dt_m \wedge d\overline{t}_1 \wedge \cdots \wedge d\overline{t}_m := \int_{X/Y} \Psi_X \cdot \frac{1}{\widetilde{h}}
\]
on the open set \( V \subset Y \), where \( \Psi_X = e^{\psi_X} \, dV_X \) (in \eqref{canonical metric}) is the \((n+m,n+m)\)-form on \( X \), and \( \widetilde{h} \) is the singular Hermitian metric on \( K_{X/Y} + D \sim_{\mathbb{Q}} f^*M_Y \) constructed earlier (see \eqref{moduli metric}). Note that the metric \( \widetilde{h} = F_i(y) \) is constant along each fiber \( X_y \), i.e., its restriction to \( X_y \) depends only on the base point \( y \in Y \). In contrast to \(\Psi_Y\), the form \(\Psi_{Y,M}\) is not a global \((m,m)\)-form on \(Y\); rather, it encodes the metric information of \(K_Y + M_Y\) that we need. Indeed, \( e^{-\psi_{Y,M}} \) may be regarded as the Hermitian metric of $K_Y +M_Y$.

We also invoke a key result from \cite[Theorem 3.1]{Tak22}, which states that the fiber integral \( F_i(y) \) admits a uniform positive lower bound on each open set \( Y_i \subset Y \). Consequently, the reciprocal \( \frac{1}{\widetilde{h}} = \frac{1}{F_i(y)} \) is uniformly bounded above on \( X_i = f^{-1}(Y_i) \). This upper bound plays an essential role in establishing the local integrability of the integrand.

 Our goal is to show that \( e^{\psi_{Y,M}} \) is a plurisubharmonic function on \( V \). To this end, we verify the mean value inequality. More precisely, we aim to prove that
$$e^{\psi_{Y,M}} \leq \frac{1}{\mathbb{B}}\int_{U}\Psi_X \cdot \frac{1}{\widetilde{h}} = \frac{1}{\mathbb{B}} \int_{\Delta} e^{\psi_{Y,M}}\, dt_1 \wedge \cdots \wedge dt_m \wedge d\overline{t}_1 \wedge \cdots \wedge d\overline{t}_m,
$$
here \( U = f^{-1}(\Delta) \), and \( \mathbb{B} \) denotes the volume of the small ball \( \Delta \subset \mathbb{C}^m \) under the standard Euclidean measure. The mean value inequality then implies the plurisubharmonicity of \( e^{\psi_{Y, M}} \), and hence shows that \( K_Y + M_Y \) admits a singular Hermitian metric with semi-positive curved.

Then we study the integration on each horizontal slice $x \times \Delta$ where $x \in X_0$ after taking the transversely holomorphic trivialization $U \sim X_0 \times \Delta$. By the Fubini theorem for fiber integrals, we compute:
$$\int_{U}\Psi_X \cdot \frac{1}{\widetilde{h}} = \int_{X_0} \int_{x\times \Delta}\Psi_X \cdot \frac{1}{\widetilde{h}}.
$$

We compute
\begin{align}
    \notag \int_{x_0 \times \Delta} \Psi_X \cdot \frac{1}{\widetilde{h}} 
    & \notag = \int_{x_0 \times \Delta} e^{\psi_X} \, dz_1 \wedge \cdots \wedge dz_n \wedge d\overline{z}_1 \wedge \cdots \wedge d\overline{z}_n \wedge dt_1 \wedge \cdots \wedge d\overline{t}_m \cdot \frac{1}{\widetilde{h}} \\
     &\notag = \int_{x_0 \times \Delta} e^{\psi_X} \cdot \frac{1}{h_K \cdot h_{\Delta_D} \cdot e^{-\varphi} \cdot |s_{D}|^2} \, dz_1 \wedge \cdots \wedge dz_n \wedge d\overline{z}_1 \wedge \cdots \wedge d\overline{z}_n \\
    \label{Uintegral} &\qquad\qquad\qquad\qquad \wedge \, dt_1 \wedge \cdots \wedge dt_m \wedge d\overline{t}_1 \wedge \cdots \wedge d\overline{t}_m.
\end{align}
Here, \( h_K \) is a smooth Hermitian metric on the relative canonical bundle \( K_{X/Y} \), which we locally write in the form
\[
h_K = e^{-\tau} \cdot dz_1 \wedge \cdots \wedge dz_n \wedge d\overline{z}_1 \wedge \cdots \wedge d\overline{z}_n,
\]
where \( \tau \in C^\infty(U) \) is a globally defined smooth function.

In the integral expression \eqref{Uintegral}, the local functions
\[
e^{\psi_X}, \quad \frac{1}{e^{-\tau} \cdot h_{\Delta_D} \cdot e^{-\varphi} |s_D|^2}
\]
are psh on \( X \), due to the pseudoeffectivity of \( K_X \) and the semi-positive curvature of the metric \( \widetilde{h} \) in \eqref{moduli metric}.
Under the transversely holomorphic trivialization \( (p, \pi) \) in \eqref{holo trivi}, the subset \( x_0 \times \Delta \subset X \) is a holomorphic submanifold. This implies that the natural inclusion \( x_0 \times \Delta \hookrightarrow X \) is holomorphic, and thus the pullbacks of the psh functions remain psh on \( x_0 \times \Delta \).

However, the local frame \( dz_1 \wedge \cdots \wedge dz_n \) for \( K_{X/Y} \) is not holomorphic along the horizontal slice \( x_0 \times \Delta \) under the deformed complex structure. According to the local expression \eqref{complex str}, we have the change of variables:
\[
dz_1 \wedge \cdots \wedge dz_n \longmapsto \bigwedge_{u=1}^n \left( dz^u - \sum_i h_i^u(z,t) \, d\overline{z}^i \right)
\]
on \( x_0 \times \Delta \). Consequently, the Hermitian metric
\[
h_K = e^{-\tau} \, dz_1 \wedge \cdots \wedge dz_n \wedge d\overline{z}_1 \wedge \cdots \wedge d\overline{z}_n
\]
also transforms accordingly to
\[
h_K = e^{-\tau} \, \left( \bigwedge_{u=1}^n \left( dz^u - \sum_i h_i^u(z,t) \, d\overline{z}^i \right) \wedge \overline{ \left( dz^u - \sum_i h_i^u(z,t) \, d\overline{z}^i \right) } \right)
\]
on the slice \( x_0 \times \Delta \). Despite this deformation, the resulting expression in \eqref{Uintegral} remains positive, and thus does not affect the plurisubharmonicity of the integrand function. Therefore, the mean value inequality holds:
\begin{align*}
& \int_{X_0} \int_{x_0 \times \Delta} 
e^{\psi_X} \cdot \frac{1}{h_K \cdot h_{\Delta_D} \cdot e^{-\varphi} |s_D|^2} \,
dz_1 \wedge \cdots \wedge dz_n \wedge d\overline{z}_1 \wedge \cdots \wedge d\overline{z}_n \\
& \qquad \wedge \, dt_1 \wedge \cdots \wedge dt_m \wedge d\overline{t}_1 \wedge \cdots \wedge d\overline{t}_m \\
& \geq \mathbb{B}  \int_{X_0} 
e^{\psi_X} \cdot \frac{1}{h_K \cdot h_{\Delta_D} \cdot e^{-\varphi} |s_D|^2} \,
dz_1 \wedge \cdots \wedge dz_n \wedge d\overline{z}_1 \wedge \cdots \wedge d\overline{z}_n \\
& = \mathbb{B} ~ e^{\psi_{Y, M}}.
\end{align*}
This yields the desired inequality:
\[
e^{\psi_{Y, M}} \leq \frac{1}{\mathbb{B}} \int_{U} \Psi_X \cdot \frac{1}{\widetilde{h}}.
\]

\subsection{Part $2$}

We have shown that the induced metric on \( K_Y + M_Y \) is semi-positive curved on the smooth locus \( Y_0 = Y \setminus Q \). Since \( B_D + C_D \) is an effective divisor, it follows that \( K_Y + M_Y + B_D + C_D \) is also semi-positive curved on \( Y_0 \). To extend this pseudoeffective metric to the whole \( Y \), it suffices to show that the corresponding plurisubharmonic (psh) function remains locally bounded above in a neighborhood of any point in \( Q \). 
To that end, we now analyze the asymptotic behavior of the Hermitian metric associated with \( K_Y + M_Y + B_D + C_D \) near the singular locus \( Q \). More precisely, we study the asymptotic behavior of the fiber integral
\[
F(y) := \int_{X_y} \Psi_X \cdot \frac{1}{\widetilde{h}} \cdot |s_E|^2,
\]
as \( y (\in Y_0) \to y_0 \in Q \), where \( s_E \) is a local canonical (a root of) section of the effective divisor \( B_D + C_D \), chosen so that \( |s_E|^2 \) captures the singularities of the metric induced by \( B_D + C_D \).

According to Theorem~\ref{thm:toroidal_resolution} in the next section, we may pass to a well-prepared birational model \( f': X' \to Y' \) of \( f \), inducing a birational morphism \( \mu: Y' \to Y \). From now on, we may assume that \( f \) is toroidal (in a neighborhood of \( X_0 \)) with respect to the pairs \( (X, P) \) and \( (Y, Q) \), and that \( f \) is equi-dimensional.

Let \( y_0 \in Q \) be a fixed point, and let \( V \subset Y \) be a coordinate neighborhood centered at \( y_0 \). We consider a locally finite open covering \( \{ U_k \} \) of \( f^{-1}(V) \), together with a partition of unity \( \{ \rho_k \} \) subordinate to \( \{ U_k \} \). Set
\[
\varphi := \Psi_X  \frac{1}{\widetilde{h}} |s_E|^2,
\]
so that for every \( y \in V \setminus Q \), we can express the fiber integral as
\[
F(y) = \sum_k \int_{X_y \cap U_k} \rho_k \varphi.
\]
Note that the sum is finite for each \( y \), since the fiber \( X_y \) is compact. Therefore, to study the singularities or asymptotic behavior of \( F(y) \) near \( y_0 \in Q \), it suffices to analyze the local integrals \( \int_{X_y \cap U_k} \rho_k \varphi \) over each chart \( U_k \). We denote by \( X_{y_0} = f^{-1}(y_0) \) the fiber over the singular point.

\noindent\textbf{Step 1: Local coordinate}. 
For consistency, we continue to follow the notational framework of~\cite{Tak22}. At each \( x_0 \in X_{y_0} \), we can find local coordinates \( U \subset (\mathbb{C}^{n+m}, x = (x_1, \ldots, x_{n+m})), V \subset (\mathbb{C}^m, t = (t_1, \ldots, t_m)) \), such that \( f|_U : U \to V \) is given by \( t = f(x), t_j = f_j(x), j = 1, \ldots, m \), with
\[
t_1 = \prod_{j=1}^{\ell_1} x_j^{a_j}, \quad t_2 = \prod_{j=\ell_1+1}^{\ell_2} x_j^{a_j}, \ldots, \quad t_m = \prod_{j=\ell_{m-1}+1}^{\ell_m} x_j^{a_j}
\]
for some \( 0 = \ell_0 < \ell_1 < \ldots < \ell_m \leq n + m \) and \( a_j \in \mathbb{Z}_{>0} \). The coordinates \( x, t \) depend on \( x_0 \), and \( Q \cap V \subset \{t_1 \cdots t_m = 0\} \) and \( P \cap U \subset \{x_1 \cdots x_{n+m} = 0\} \). We can write as \( \ell_i = \ell_{i-1} + n_i + 1 \) with \( n_i \in \mathbb{Z}_{\geq 0} \) and $\sum_{i=1}^{m+1} n_i = n$,  where $n_{m+1} := n + m - \ell_m$. We denote by 
$$y_j = x_{\ell_j}, b_j = a_{\ell_j}$$
for \( j = 1, \ldots, m \), and by
\[
\textbf{x}_j = (x_{\ell_{j-1}+1}, \ldots, x_{\ell_{j-1}+n_j}), \quad \textbf{a}_j = (a_{\ell_{j-1}+1}, \ldots, a_{\ell_{j-1}+n_j}),
\]
\[
\textbf{x}_{m+1} = (x_{\ell_m+1}, \ldots, x_{n+m}).
\]
We will use a multi-index convention. For example \( t_j = \textbf{x}_j^{a_j} y_j^{b_j} \).

\noindent\textbf{Step 2: Estimate of $\Psi_X \cdot |s_E|^2$.}
We choose local frames \( dx = \wedge_{i=1}^{n+m} dx_i \) of \( K_X |_U \) and \( dt = \wedge_{j=1}^m dt_j \) of \( K_Y |_V \) respectively. From the relation 
\[
\frac{dt_j}{t_j} = \sum_{i=\ell_{j-1}+1}^{\ell_{j-1}+n_j} a_i \frac{dx_i}{x_i} + b_j \frac{dy_j}{y_j},
\]
we deduce that
\[
\wedge_{i=\ell_{j-1}+1}^{\ell_j} dx_i = \frac{y_j}{b_j t_j} (\wedge_{i=\ell_{j-1}+1}^{\ell_j} dx_i) \wedge dt_j,
\]
in other term \( d\textbf{x}_j \wedge dy_j = \frac{y_j}{b_j t_j} d\textbf{x}_j \wedge dt_j \). Therefore, a local frame for the relative canonical bundle \( \frac{dx}{dt} \) of \( K_{X/Y} \) on \( U \) can be expressed as
\[
\frac{dx}{dt} = \bigwedge_{j=1}^m \left( \frac{y_j}{b_j t_j} d\textbf{x}_j \right) \wedge d\textbf{x}_{m+1},
\]
which satisfies the identity \( \frac{dx}{dt} \wedge f^* dt = \pm dx \).

We study the local integral of 
\begin{equation} \label{integral-s}
\varphi = \Psi_X \cdot \frac{1}{\widetilde{h}} \cdot |s_E|^2 = e^{\psi_X} \left| \frac{dx}{dt} \wedge f^*dt \right|^2 \cdot \frac{1}{\widetilde{h}} \cdot |s_E|^2.
\end{equation}
We can suppose \( V \setminus Y_0 = \{ t_1 \cdots t_q = 0 \} \) for some \( 0 \leq q \leq m \). We can suppose \( Q \cap V = V \setminus Y_0 \). We let \( Q_j \cap V = \{ t_j = 0 \} \) for \( 1 \leq j \leq q \), and \( P_i \cap U = \{ x_i = 0 \} \) for \( 1 \leq i \leq \ell_q \) (such that \( P_i \) is \( f \)-vertical and appears in \( D \)).

For notational simplicity, we suppose \( q = m \) (and hence \( \ell_q = \ell_m \)). Other cases \( 0 \leq q \leq m-1 \) can be deduced from the case \( q = m \). In particular, \( P_i \) is \( f \)-vertical if and only if \( 1 \leq i \leq \ell_m \). We write
\[
D^h_{-} = \operatorname{div} \left( \prod_{i > \ell_m, -d_i > 0} x_i^{-d_i} \right), \quad 
D^v_{-} = \operatorname{div} \left( \prod_{i \leq \ell_m, -d_i > 0} x_i^{-d_i} \right).
\]
We estimate the integrand \eqref{integral-s}. Shrinking \( V \) and \( U \) a bit, we may suppose
\[
\left| \frac{dx}{dt} \right|^2 = \left| \frac{y}{t} \right|^2 = \prod |y_i|^{2(1 - a_{\ell_i})} \prod \frac{1}{|x_j|^{2a_j}},
\]
and by \eqref{ej}
\[
|s_E|^2 = \prod_j |t_j|^{2e_j} \geq \prod |x_i|^{2(a_i - 1)}.
\]
Recall that
\[
e_j = \beta_j + \gamma_j = \max_i \left\{ \frac{d_i + a_{ji} - 1}{a_{ji}} \right\} + \max_i \left\{ \frac{-d_i}{a_{ji}} \right\} 
\geq \frac{d_i + a_{ji} - 1}{a_{ji}} + \frac{-d_i}{a_{ji}} = \frac{a_{ji} - 1}{a_{ji}} \geq 0.
\]
We also have the local expression \( t_j = \mathbf{x}_j^{a_j} y_j^{b_j} \).
Hence, the singular part of \( \Psi_X \cdot |s_E|^2 \) is locally given by
\begin{equation}\label{star1}
\left| \frac{dx}{dt} \right|^2 \cdot \prod_j |t_j|^{2e_j} \sim \left| \frac{y}{t} \right|^2 \cdot \prod_j |t_j|^{2e_j}. \quad~~~ \textbf{($\bigstar$)} 
\end{equation}

\textbf{Step 3: Estimate of $\widetilde{h}=F_i(y)$.}
For simplicity, here we suppose that $B_D=0$ since here we only deal with $M_Y$.
Then we estimate the value of $\widetilde{h} = F_i(y)$, which is constant on each fiber; moreover, it is equal (see \eqref{Fi})
\[
F_i(y) = \int_{X_y} g_i c_n \sigma_y \wedge \bar{\sigma}_y.
\]
We may suppose
\begin{align*}
    |\sigma_U|^2 &= A_1 \cdot \prod_{i,-d_i > 0} |x_i|^{2\lceil - d_i \rceil} \quad(\sim |\lceil{D_-}\rceil|^2), \\
g &= A_2 \cdot \left( \prod_{i,-d_i > 0} |x_i|^{2(\lceil -d_i\rceil  + d_i)} \right)^{-1}  \quad\left( \sim \frac{1}{ |\lceil {D_-}\rceil - D_-|^2} \right)
\end{align*}
on \( U \), where \( A_1, A_2 \) are \( C^\infty \)-smooth and positive functions on a neighborhood of \( U \). 
Its symbolic calculus is
\begin{align} \label{star2}
 \notag (g c_n \sigma \wedge \overline{\sigma})|_{X_y \cap U}  &\sim \frac{1}{ |\lceil D_- \rceil - D_-|^2} |\lceil D_- \rceil|^2 \left| \frac{d x}{d t} \right|^2 \\
\notag &\sim |D_-|^2 \left| \frac{y}{t} \right|^2 \prod_{j=1}^{m+1} |d \textbf{x}_j|^2 \\
&= \left| \frac{D^v_{-}y}{t}  \right|^2 \prod_{j=1}^m |d \textbf{x}_j|^2 \cdot |D_-^h|^2 |d \textbf{x}_{m+1}|^2. \quad~~~\textbf{($\bigstar\bigstar$)}
\end{align}

Then, for any \( t \in V \setminus Q \),
\[
\quad (g c_n \sigma \wedge \overline{\sigma})|_{X_y \cap U} = A_1 \cdot \prod_{j=1}^m b_j^{-2} \cdot A_2 \cdot \bigwedge_{j=1}^{m+1} W_j
\]
holds. Here we meant \( A_i|_{X_y \cap U} \) and \( |y_j| = |t_j / \textbf{x}_j^{a_j}|^{1/b_j} \), and here

\begin{align*}
&W_j = (|y_j|^2)^{-d_{\ell_j}+1-a_{\ell_j}} \prod_{i=\ell_{j-1}+1}^{\ell_{j-1}+n_j} (|x_i|^2)^{-d_i - a_i} c_{n_j} d x_j \wedge d \overline{x}_j, \quad j=1,\ldots,m, \\
&W_{m+1} = \prod_{i>\ell_m,d_i<0} |x_i|^{-2d_i} c_{n_{m+1}} d x_{m+1} \wedge d \overline{x}_{m+1}.
\end{align*}

 We note, as a consequence of \( B_D = 0 \), that for every \( j \in \{1,\ldots,m\} \) and \( i \in \{\ell_{j-1}+1,\ldots,\ell_j\}, -d_i - a_i \geq -1 \) holds, in particular \( -d_{\ell_j} + 1 - a_{\ell_j} \geq 0 \) holds. In view of these estimates, we then introduce a (possibly singular) semi-positive \((n,n)\)-form
\[
\psi = \bigwedge_{j=1}^m \frac{c_{n_j} d x_j \wedge d \overline{x}_j}{|x_j|^2} \wedge c_{n_{m+1}} d x_{m+1} \wedge d \overline{x}_{m+1}
\]
on \( (\mathbb{C}^{n+m}, x) \), where \( |x_j|^2 = \prod_{i=\ell_{j-1}+1}^{\ell_{j-1}+n_j} |x_i|^2 \). We then have a constant \( A > 0 \) (independent of \( t \)) such that
\[
0 \leq (g c_n \sigma \wedge \overline{\sigma})|_{X_t \cap U} \leq A \psi|_{X_t \cap U}
\]
holds for any \( t \in V \setminus Q \). Here we disregarded \( |D_-^h|^2 \)-part in the bottom line in \eqref{star2}. We evaluate \( \int_{X_t \cap U} \widetilde{R} \psi \) for a larger domain \( \widetilde{U} (\supset U) \) in \( \mathbb{C}^{n+m} \) and a cut-off function \( \widetilde{R} \) with Supp \( \widetilde{R} \subset \widetilde{U} \). According to Takayama \cite[Theorem $3.1$]{Tak22}, the function \( F_i \) admits the following estimate:
\[
c_1 \leq F_i \leq c_2 \prod_{j=1}^m (-\log|t_j|)^n,
\]
for some positive constants \( c_1 \) and \( c_2 \).

\textbf{Step 4: Comparison.}
As the equality \((\bigstar\bigstar)\) shows, we have
\[
\Psi_X \cdot |s_E|^2 \sim \left| \frac{y}{t} \right|^2 \cdot \prod_j |t_j|^{2e_j}.
\]
Note that \( e_j = \beta_j + \gamma_j = \max_i \left\{ \frac{d_i + a_{ji} - 1}{a_{ji}} \right\} + \max_i \left\{ \frac{-d_i}{a_{ji}} \right\} \), where \( \beta_j \) represents the coefficient of \( B_D \), and \( \gamma_j \) represents the maximal part of \( D_{-}^v \) (i.e., taking the maximum among all coefficients).

On the other hand, in the formula \((\bigstar)\), we have
\[
(g c_n \sigma \wedge \overline{\sigma})|_{X_y \cap U} \sim \left| \frac{D_{-}^v y}{t} \right|^2 \prod_{j=1}^m |d \mathbf{x}_j|^2 \cdot |D_{-}^h|^2 \cdot |d \mathbf{x}_{m+1}|^2.
\]
We observe that there is no \( f^* B_D \) contribution here, due to the assumption \( B_D = 0 \) in Step 3 (but we consider it here). Moreover, it is clear that \( f^* (\gamma_j Q_j) \geq D_{-}^v \) (The effectiveness of the divisor increases with the size of the positive exponent \( a_j \) in \( x_j^{a_j} \)). We do not consider the \( D^h_- \) part, as it is horizontal and does not affect the fiber integral.

Therefore, we obtain the inequality
\[
\Psi_X \cdot |s_E|^2 \leq C \cdot (g c_n \sigma \wedge \overline{\sigma})
\]
for some constant \( C > 0 \). It follows that
\[
\int_{X_y} \Psi_X \cdot |s_E|^2 \leq C \int_{X_y} (g c_n \sigma \wedge \overline{\sigma}) = C F_i(y) = c \cdot \widetilde{h}(y),
\]
which implies that
\[
F(y) := \int_{X_y} \Psi_X \cdot \frac{1}{\widetilde{h}} \cdot |s_E|^2 = \frac{1}{\widetilde{h}(y)} \int_{X_y} \Psi_X \cdot |s_E|^2
\]
remains bounded as \( y \to y_0 \) for any point \( y_0 \in Q \).

By the extension property of bounded plurisubharmonic functions, the line bundle \( K_X + M_Y + B_D + C_D \) admits a Hermitian metric with semi-positive curvature on the whole of \( Y \). This completes the proof of the Theorem \ref{claim}.

\section{The proof of the Schnell conjecture}
We now prove Theorem~\ref{schnell-3} in this section, following the same strategy as in the proof of Theorem~\ref{claim}. As previously mentioned, under the assumption of the non-vanishing conjecture, the Campana--Peternell conjecture admits the following reformulation:  

\begin{theorem} \label{schnell-2}
Let \( f: X \to Y \) be an algebraic fiber space between smooth projective varieties, and suppose that the general fiber \( F \) satisfies \( \kappa(F) = 0 \). If there exists an ample divisor \( H \) on \( Y \) such that \( m_0K_X - f^*H \) is pseudo-effective for some \( m_0 \geq 1 \), then \( \kappa(X) = \dim Y \).
\end{theorem}


To make this precise, we recall some notation. First, following \cite[\S 5]{Sch22}, we observe that the conjecture is invariant under taking birational models of \( X \) and \( Y \). By taking suitable resolutions, we may and will assume that there exists an SNC divisor \( Q = \sum_{i \in I} Q_i \) on \( Y \) such that \( f \) is smooth over the open subset \( Y_0 = Y \setminus Q \), and that the reduced inverse image \( f^{-1}(Q)_{\mathrm{red}} = \sum_{j \in J} P_j \) is an SNC divisor on \( X \).
Following \cite{FM00, Kol07}, we may write the canonical divisor \( K_X \) as:
\begin{equation} \label{eq:canonical_decomp_2}
K_X + D \sim_{\mathbb{Q}} f^*(K_Y + B_D + M_Y),
\end{equation}
where the decomposition satisfies the following properties:

\begin{enumerate} \label{three-cond}
\item \( f_* \mathscr{O}_X (\lfloor iD_{-} \rfloor) = \mathscr{O}_Y \) for all \( i > 0 \);
\item \( \operatorname{codim}_Y f(\operatorname{Supp} D_{+}) \geq 2 \);
\item \( H^0(X, mK_X) = H^0(Y, m(K_Y + B_D + M_Y)) \) for all sufficiently divisible integers \( m \) (see \cite[Theorem 4.5]{FM00}).
\end{enumerate}

Since this conjecture is sensitive to taking finite coverings, we cannot use the so-called weak semistable reduction of \( f \). Instead, we apply a resolution of singularities theorem for morphisms due to Abramovich and Karu~\cite{AK00}.


\begin{theorem} \label{thm:toroidal_resolution}
Let \( f:X \to Y \) be a surjective morphism of smooth projective varieties with connected fibers and \( Z \) a closed subset of \( X \). Then there exists:
\begin{itemize}
\item A quasi-smooth projective toroidal variety \((X',B')\),
\item A smooth projective toroidal variety \((Y',C')\),
\item A projective morphism \( f':X'\to Y' \) with connected fibers, and
\item Projective birational morphisms \( v_X:X'\to X \), \( v_Y:Y'\to Y \)
\end{itemize}
such that \( v_Y\circ f'=f\circ v_X \) and which satisfy the following conditions:
\begin{itemize}
\item \( f':(X',B') \to (Y',C') \) is a toroidal morphism.
\item All the fibers of \( f' \) have the same dimension (\( f' \) is flat in fact).
\item \( v_X^{-1}(Z) \subset B' \).
\end{itemize}
\end{theorem}

We call the morphism \( f':(X',B') \to (Y',C') \) a \textit{well-prepared birational model} of \( f:X\to Y \). A toroidal variety \((X,B)\) is said to be:
\begin{itemize}
\item \textit{smooth} if \( X \) is smooth and \( B \) has only normal crossings.
\item \textit{quasi-smooth} if there exists a local toric model of each point which is a quotient of a smooth toric local model by a finite abelian group (some people say that \( X \) is a \( V \)-manifold and \( B \) is \( V \)-normal crossing).
\end{itemize}

Finally, we begin the proof of Schnell's conjecture, following the proof of Theorem~\ref{claim} with some modifications.
\begin{proof} [Proof of Theorem \ref{schnell-2}]
We consider the canonical bundle formula in \eqref{eq:canonical_decomp_2}. Here, we only consider the negative part $D_-=D^h_- + D^v_-$ of $D$, as given by $(2)$ above. For a fixed component $Q_j$ in $Q$, let $f^*Q_j = \sum_i a_{ji}P_i$. Since $D^v_- = \sum (-d_i)P_i$, and \( f_* \mathscr{O}_X (\lfloor iD_{-} \rfloor) = \mathscr{O}_Y \) for all \( i > 0 \), which yields that at least one of coefficient $d_i$ in $P_i$ which in the preimage of $Q_j$ is $0$. This means that $B_D$ here is effective, whereas the $B_D$ in Theorem~\ref{claim} need not be effective.


As the proof of Theorem \ref{claim} shows, the pseudo-effectiveness of $K_X$ implies that $K_Y+M+B_D +C_D$ is pseudo-effective too. We consider the variant formula 
$$
K_X + D -\delta f^*H \sim_{\mathbb{Q}} f^*(K_Y + B_D + M_Y-\delta H).
$$
for ample divisor $H$ in $Y$ and here $\delta$ is some sufficient small positive rational number such that $K_X -\delta f^*H$ is pseudoeffective by the assumption in Theorem \ref{schnell-2}. 
Let \( e^{-\psi_X} \) be the singular Hermitian metric on \( K_X \). By assumption, there exists a smooth positive metric \( e^{-\psi_H} \) of \( H \) on $Y$ such that 
\[
\psi_X - \delta f^*\psi_H = \psi_{X,H}
\]
for some psh weight function \( \psi_{X,H} \). Hence, we have 
\(
\psi_X = \psi_{X,H} + \delta f^*\psi_H.
\)
As in \eqref{canonical metric}, we define
\[
\Psi_X := e^{\psi_{X,H} + \delta f^*\psi_H} \, dz_1 \wedge \cdots \wedge dz_n \wedge d\overline{z}_1 \wedge \cdots \wedge d\overline{z}_n \wedge dt_1 \wedge \cdots \wedge dt_m \wedge d\overline{t}_1 \wedge \cdots \wedge d\overline{t}_m.
\]
By the same argument as in the proof of Theorem~\ref{claim}, it follows that the divisor 
\[
K_Y + B_D + C_D + M_Y - \delta H
\]
is pseudo-effective. Note that \(f^*\psi_H\) is constant along each fiber and does not affect the fiber integrals.
We conclude that \( K_Y + B_D + C_D + M_Y \) is big. The divisor \( C_D \) constitutes an obstruction to directly establishing the bigness of \( K_Y + B_D + M_Y \). We now investigate the bigness of \( K_Y + B_D + M_Y \).

Now that we know \(B_D\) is effective, we first construct a semi-positively curved metric on \(K_Y+M_Y\) over \(Y_0\). Adding \(B_D\) preserves pseudoeffectivity on \(Y_0\). The main difficulty lies in extending this metric to \(Y\) without involving \(C_D\). To address this, we modify the metric of \(M_Y\) on \(Y_0\) to obtain a new metric on \(K_Y+B_D+M_Y\), which can then be extended to the whole of \(Y\) by Proposition~\ref{Tak24}.

Since \( C_D \) is supported on the exceptional locus \( Q \), we now consider the following formula on the open subset \( Y_0 := Y \setminus Q \):
\[
K_X + D - \delta f^*H \sim_{\mathbb{Q}} K_X - D^h_- - \delta f^*H \sim_{\mathbb{Q}} f^*(K_Y + M_Y - \delta H).
\]
In the first part of the proof of Theorem~\ref{claim}, we showed that if \( K_X - \delta f^*H \) is pseudo-effective, then \( K_Y + M_Y - \delta H \) is also pseudo-effective on \( Y_0 \), and the metric on \( K_X \) descends to a singular Hermitian metric on \( K_Y \). 
Now over $Y_0$, as \eqref{moduli line}, we have 
$$
K_{X/Y} + \Delta_D - \lceil D^h_{-} \rceil = K_{X/Y} -D^h_- \sim_{\mathbb{Q}} f^*M_Y.
$$
Moreover, in formula~\eqref{moduli metric}, the global singular Hermitian metric \( h_K h_{\Delta_D} e^{-\varphi} \) on \( K_{X/Y} + \Delta_D \) has the property that the Lelong number of \( \varphi \) at each point is greater than or equal to that of \( \lceil D_- \rceil \) (namely the multiplicity), and in particular, greater than \( \lceil D^v_- \rceil \).

Hence, we can endow \( f^*M_Y \) with a semi-positive singular Hermitian metric \( h' \) induced from the metric \( h_K h_{\Delta_D} e^{-\varphi} \) on \( K_{X/Y} + \Delta_D \) over \( f^{-1}(Y_0) \), using the relation
\begin{equation} \label{moduli line 2}
K_{X/Y} + \Delta_D - \lceil D^h_- \rceil \sim_{\mathbb{Q}} f^*M_Y.
\end{equation}
This metric \( h' \) is defined explicitly by
\begin{equation} \label{moduli metric 2}
h' := \frac{\widetilde{h}}{|s_{\lceil D^v_- \rceil}|^2} 
= \frac{h_K \cdot h_{\Delta_D} \cdot e^{-\varphi} \cdot |s_{\lceil D_- \rceil}|^2}{|s_{\lceil D^v_- \rceil}|^2} 
= \frac{F_i}{|s_{\lceil D^v_- \rceil}|^2},
\end{equation}
where \( \widetilde{h} := h_K \cdot h_{\Delta_D} \cdot e^{-\varphi} \cdot |s_{\lceil D_- \rceil}|^2 \) is the metric on \( K_{X/Y} + \Delta_D -\lceil D_- \rceil = K_{X/Y} +D\) as in \eqref{moduli metric}.

For a local non-vanishing section \( \sigma \in H^0(\Omega, M_Y) \) for \( \Omega \subset Y_0 \), we see that
\[
|\sigma|_{h',x}^2 : X_y \to \mathbb{R}_{\geq 0} \cup \{+\infty\}
\]
has to be a constant function since the only global psh functions on a compact space are constants. We define \( |\sigma|_{h'',y} := |\sigma|_{h',x} \) for \( y = f(x) \). This gives a positively curved singular metric on \( M_Y \) over \( Y_0 \). 
To our end, we now analyze the asymptotic behavior of the Hermitian metric associated to \( K_Y + M_Y + B_D  \) near the singular locus \( Q \). More precisely, we study the asymptotic behavior of the fiber integral
\[
F'(y) := \int_{X_y} \Psi_X \cdot \frac{1}{h'} \cdot |s_B|^2,
\]
as \( y (\in Y_0) \to y_0 \in Q \), where \( s_B \) is a local canonical (a root of) section of the effective divisor \( B_D \) similar to \eqref{root}.
  
Similar to the second part (Step 4) of the proof of Theorem \ref{claim}, we can show that \( F'(y) \) remains bounded as \( y \to y_0 \) for any point \( y_0 \in Q \). Indeed, in contrast to formula~\eqref{star1} ($\bigstar$), in this case we have
\[
\Psi_X \cdot |s_B|^2 \sim \left| \frac{y}{t} \right|^2 \cdot \prod_j |t_j|^{2\beta_j}.
\]
Note that \( \beta_j  = \max_i \left\{ \frac{d_i + a_{ji} - 1}{a_{ji}} \right\}  \), where \( \beta_j \) represents the coefficient of \( B_D \) which is non-negative. There exists a constant \( C > 0 \). such that
\[
\int_{X_y} \Psi_X \cdot |s_B|^2 \leq C \int_{X_y} (g c_n \sigma' \wedge \overline{\sigma'}) = C h'(y),
\]
where \(\sigma'\) is simply the canonical section of \(s_{\lceil D^h_- \rceil}\), rather than \(s_{\lceil D_- \rceil}\).
Indeed, in contrast to formula~\eqref{star2} \((\bigstar\bigstar)\), we now obtain
\[
(g c_n \sigma' \wedge \overline{\sigma'})|_{X_y \cap U} \sim \left| \frac{y}{t} \right|^2 \prod_{j=1}^m |d \mathbf{x}_j|^2 \cdot |D_{-}^h|^2 \cdot |d \mathbf{x}_{m+1}|^2.
\]
This inequality implies that $F'(y)$ remains bounded as \( y \to y_0 \) for any point \( y_0 \in Q \).

We interpret $F'(y)$ as the norm of sections of $-(K_Y + M_Y + B_D)$ over $Y_0$. 
The Proposition \ref{Tak24}, combined with the extension property of bounded plurisubharmonic functions, shows that the line bundle $K_X + M_Y + B_D$ admits a semi-positively curved Hermitian metric globally on $Y$. 
Since $K_X - \delta f^*H$ is positive, it follows that $K_Y + M_Y + B_D$ is big, which completes the proof.
\end{proof}

\section{Miscellaneous results}
  
We prove the following results using our new methods, which are analytically natural. These results concern the smooth descent of the positivity of the canonical bundle. For related developments, the reader may consult \cite{PS23, Par23}. 

\begin{theorem} \cite[Proposition G]{PS23} \label{G}
Let \(f \colon X \to Y\) be a smooth algebraic fiber space with connected fibers, and denote by \(F\) a general fiber. 
Assume that \(\kappa(X) \geq 0\).
Then $ K_Y$  is pseudoeffective.
\end{theorem}

\begin{proof}
In this case, we have the canonical bundle formula
\[
K_X - K_{X/Y} = f^{*}(K_Y).
\]
The proof proceeds as in Part~1 of the proof of Theorem~\ref{claim}, with $K_{X/Y}$ playing the role of $M_Y$ in that argument.

We denote the $\nu$-th relative Bergman kernel metric of $-\nu K_{X/Y}$ by $B_{\nu}$. 
For $y = f(x)$, we define
\begin{equation} \label{berg}
B_{\nu}(x) = \sup \left\{\, c\, u \wedge \bar{u}(x) \,\middle|\, 
u \in H^0(X_y, \nu K_{X_y}),\ \|u\|_{\nu} \le 1 \right\}.
\end{equation}

In our situation, take a section $s \in H^0(X, K_X)$. 
We define a metric on $K_X$ by
\[
e^{-\psi_X} := \frac{1}{|s|^2}.
\]
Let $h_{\nu} := \frac{1}{B_{\nu}}$ be the relative Bergman-type metric on $\nu K_{X/Y}$.
Recall that
\[
\Psi_X := e^{\psi_X} \, 
dz_1 \wedge \cdots \wedge dz_n \wedge 
d\overline{z}_1 \wedge \cdots \wedge d\overline{z}_n \wedge
dt_1 \wedge \cdots \wedge dt_m \wedge
d\overline{t}_1 \wedge \cdots \wedge d\overline{t}_m
\]
defines a global $(n+m,n+m)$-form on $X$.

Instead, we consider the following:
\begin{equation} \label{psi-l}
\Psi := \int_{X/Y}  \frac{1}{h_1} \, h_1 e^{\psi_X} \, 
dz_1 \wedge \cdots \wedge dz_n \wedge 
d\overline{z}_1 \wedge \cdots \wedge d\overline{z}_n \wedge
dt_1 \wedge \cdots \wedge d\overline{t}_m.
\end{equation}
If we set
\[
\Psi := e^{\psi_Y} \, dt_1 \wedge \cdots \wedge dt_m \wedge d\overline{t}_1 \wedge \cdots \wedge d\overline{t}_m,
\]
then \(e^{-\psi_Y}\) defines a metric on \(K_Y\), and, as part $1$ of the proof of the main theorem, one can show that \(e^{\psi_Y}\) (hence $\psi_Y$) is psh.

Note that \(h_1\), as the metric on \(K_{X/Y}\), plays the role of \(M_Y\) in previous arguments, reflecting the variation of the complex structure. 
By the definition of the Bergman metric,
\[
h_1e^{\psi_X} =h_1 |s|^2 = \frac{|s|^2}{B_1},
\]
which is bounded, and it is also the ratio of two holomorphic sections. Therefore, it must be plurisubharmonic. As we just said, \(1/h_1\) defines a pseudoeffective metric, and hence is plurisubharmonic.
These observations imply that \(e^{\psi_Y}\) is plurisubharmonic when considering the fiber integral.
\end{proof}

Using the same method, we can also prove the following results in \cite{PS23}.
\begin{corollary}
Let \( f: X \to Y \) be a smooth projective morphism of smooth projective varieties. If \(X\) is of general type, then \(Y\) is also of general type.
\end{corollary}

\begin{theorem} \cite[Theorem A]{PS23} \label{A}
Let \(f \colon X \to Y\) be a smooth algebraic fiber space with connected fibers, and denote by \(F\) a general fiber. 
Assume that \(\kappa(F) \geq 0\).
Then
\[
Y \text{ is of general type }
\quad \Longleftrightarrow \quad
\kappa(X) = \kappa(F) + \dim Y .
\]
\end{theorem}

\begin{proof}
    
We start by considering a general construction. Let $A$ be a very ample line bundle on $Y$. 
According to the easy addition formula, the condition
\[
\kappa(X)=\kappa(F)+\dim Y
\]
is equivalent to the existence of an integer $l\ge1$ such that
\begin{equation*}\label{eq:nonzero-section}
H^0\!\bigl(X,\ \omega_X^{\otimes l}\otimes f^*A^{-1}\bigr)\neq 0.
\end{equation*}

\textbf{Step 1.} If \(K_Y\) is big, we can choose a sufficiently large integer \(l\) such that \(l K_Y - A\) is pseudoeffective. Consider
\[
l K_X - f^* A = l K_{X/Y} + f^*(l K_Y - A).
\]
Since for the generic fiber \(F\) we have \(\kappa(F) \ge 0\), the relative canonical bundle \(K_{X/Y}\) admits a Bergman metric that is pseudoeffective. Combining this with Theorem~\ref{schnell-1} below, we conclude that the pseudoeffectiveness of \(l K_X - f^* A\) implies that \(l K_X - f^* A\) is effective.

\textbf{Step 2.} Assume that \(l K_X - f^* A = D\) is effective. Then \(lK_X\) is effective, and one of its sections can be written as
\[
s = s_D \cdot s_A,
\]
where \(s_D\) is the canonical section of \(D\) and \(s_A\) is a section of \(f^* A\). 
Let \(h_A\) denote a (semi-)positive metric on \(f^* A\), obtained as the pullback of a strictly positive metric on \(A\) over \(Y\).

We denote the $l$-th relative Bergman kernel metric of $-(l K_{X/Y}-f^*A)$ by $B_{l}$. 
For $y = f(x)$, we define
\begin{equation} \label{bergl}
B_{l}(x) = \sup \left\{\, c\, u \wedge \bar{u}(x) \,\middle|\, 
u \in H^0(X_y, (lK_{X_y}-f^*A)),\ \|u\|_{l} \le 1 \right\}.
\end{equation}
Note that $lK_{X_y}-f^*A \sim lD_y$ for generic $y\in Y$.
One can obtain a metric of $K_X$ as follows
\[
e^{-\psi_X} := \frac{1}{|s_D|^{2/l}} \, h_A^{1/l},
\]
and let \(h'_l := \frac{1}{B_l}\) be the Bergman metric on \(l K_{X/Y} - f^* A\). 
This induces a pseudoeffective metric on \(K_{X/Y}\) given by
\[
h_1 := \bigl(h'_l \cdot h_A\bigr)^{1/l}.
\]

Now we have the fiber integral
\[
\Psi_1 := \int_{X/Y}\frac{1}{h_1}  \, h_1 e^{\psi_X} \, 
dz_1 \wedge \cdots \wedge dz_n \wedge d\overline{z}_1 \wedge \cdots \wedge d\overline{z}_n \wedge dt_1 \wedge \cdots \wedge d\overline{t}_m.
\] 
If we instead write
\[
\Psi_1 := e^{\psi_Y} \, dt_1 \wedge \cdots \wedge dt_m \wedge d\overline{t}_1 \wedge \cdots \wedge d\overline{t}_m,
\] 
then \(e^{-\psi_Y}\) defines a metric on \(K_Y\). As the proof of Theorem \ref{G}, one can show that \(e^{\psi_Y}\) is strictly plurisubharmonic, which implies that \(K_Y\) is big.
\end{proof}


For a general surjective morphism, which need not be smooth, we present the following result of Viehweg--Zuo from our perspective.

\begin{theorem} \cite[Theorem 0.2]{VZ01} \label{VZ}
Let \(X\) be a complex projective manifold of non\text{-}negative Kodaira dimension. Then any surjective morphism \(f: X \to \mathbb{P}^1\) has at least three singular fibres.
\end{theorem}

\begin{proof}
We use the same notations as in the proof of Theorem \ref{G}.
As in \eqref{psi-l}, we consider the fiber integral
\[
\Psi_1 := \int_{X/Y} h_1 \frac{1}{h_1} \, e^{\psi_X} \, 
dz_1 \wedge \cdots \wedge dz_n \wedge d\overline{z}_1 \wedge \cdots \wedge d\overline{z}_n \wedge dt_1 \wedge \cdots \wedge d\overline{t}_m,
\] 
which defines a metric on \(K_Y\).
The integrand \(h_1 \frac{1}{h_1} e^{\psi_X}\) is bounded, but the volume form
\[
dz_1 \wedge \cdots \wedge dz_n \wedge d\overline{z}_1 \wedge \cdots \wedge d\overline{z}_n \wedge dt_1 \wedge \cdots \wedge d\overline{t}_m
\]
may blow up near the singular fibers of \(f\). We denote the singular locus by \(Q = \sum Q_i\), where the \(Q_i\) are points in \(\mathbb{P}^1\).

Let 
\begin{align*}
& f^* Q_j = \sum_i a_{ji} P_i, \quad a_{ji} \in \mathbb{Z}_{>0},\\
& \gamma_i = \frac{a_{ji}-1}{a_{ji}} \in \mathbb{Q}_{<1} \quad \text{if } f(P_i) = Q_j,\\
& \beta_j = \max \{ \gamma_i \mid f(P_i) = Q_j \} \in \mathbb{Q}_{<1},\\
& W := \sum_j \beta_j Q_j.
\end{align*}
Hence, we consider the corresponding fiber integral
\[
\Psi'_1 := \int_{X/Y} h_1 \frac{1}{h_1} \, e^{\psi_X} |s_W|^2 \, 
dz_1 \wedge \cdots \wedge dz_n \wedge d\overline{z}_1 \wedge \cdots \wedge d\overline{z}_n \wedge dt_1 \wedge \cdots \wedge d\overline{t}_m,
\]
where \(|s_W|^2\) denotes the canonical section of \(W\) on $Y$. If we write
\[
\Psi'_1 := e^{\psi_{Y,W}} \, dt_1 \wedge \cdots \wedge dt_m \wedge d\overline{t}_1 \wedge \cdots \wedge d\overline{t}_m,
\] 
then $e^{-\psi_{Y,W}}$ defines a metric on \(K_Y + W\), which is pseudoeffective.

Since \(K_Y = \mathcal{O}(-2)\) and all \(\beta_i < 1\), it follows that \(Q\) must contain at least three distinct points. This completes the proof.
\end{proof}

Lastly, we will show how the Ohsawa--Takegoshi theorem can be used to study the existence of sections of the canonical bundle in an algebraic fiber space.

\begin{theorem} \cite[Theorem 12.1]{Sch22} \label{schnell-1}
    Let $f: X \to Y$ be an algebraic fiber space, meaning $X$ and $Y$ are smooth projective varieties and $f$ is surjective with connected fibers. Suppose that $K_Y$ is pseudoeffective.
 Let $F$ be a generic fiber, which satisfies $\kappa(F) \geq 0$. Let $H$ be an ample divisor on $Y$. If $mK_X - f^*H$ is pseudo-effective for some $ m\geq 1$, then $mK_X - f^*H$ becomes effective for $m$ sufficiently large and divisible.

\end{theorem}

The Ohsawa-Takegoshi type theorems are fundamental tools in complex geometry. In this paper, we will employ the following, very general version.
\begin{theorem}[Ohsawa--Takegoshi--Manivel extension] \cite[Theorem $1$]{Man93} \cite[Theorem 5.1]{BP08} \label{ot1} 
Let $X$ be a projective or Stein $n$-dimensional manifold possessing a K\"ahler metric $\omega$, and let $E$ be a Hermitian holomorphic vector bundle of rank $k$ over $X$, and $s$ a global holomorphic section of $E$. The line bundle $L$ has a possible singular metric $h_L$. Assume that $s$ is generically transverse to the zero section, and let
\[ Z = \{ x \in X \mid s(x) = 0, \Lambda^k ds(x) \neq 0 \}, \quad \dim Z = n - k. \]
Moreover, we assume there is a continuous function $\alpha > 1$ such that the following inequalities hold everywhere on $X$:
\begin{align*}
(a) & \quad \sqrt{-1}\Theta_{h_L}(L) + \sqrt{-1} \partial \bar{\partial} \log |s|^2 \geq 0 \quad \text{as current on}\quad X,\\
(b) & \quad \sqrt{-1}\Theta_{h_L}(L) + \sqrt{-1} \partial \bar{\partial} \log |s|^2 \geq \alpha^{-1} \frac{\{\sqrt{-1} \Theta(E) s, s\}}{|s|^2}, \\
(c) & \quad |s|^2 \leq e^{-\alpha}.
\end{align*}

Then for every holomorphic section $f$ over $Z$ with values in the line bundle $\Lambda^n T^*_X \otimes L$ (restricted to $Z$), such that
\[ \int_Z \frac{|f|_{h_L}^2}{|\Lambda^k (ds)|^2} \, dV_\omega < +\infty, \]
there exists a holomorphic section $\tilde{f}$ over $X$ with values in $\Lambda^n T^*_X \otimes L$ satisfies $\tilde{f}|_Y = f$, and
\[ \int_X \frac{|\tilde{f}|_{h_L}^2}{|s|^{2k} (- \log |s|)^2} \, dV_{X, \omega} \leq C \int_Y \frac{|f|_{h_L}^2}{|\Lambda^k (ds)|^2} \, dV_{Y, \omega}, \]
where $C$ is a universal constant.
\end{theorem}

In the context of Theorem \ref{schnell-1}, we can choose the general fiber \(F\) so that \(f\) is smooth in a small neighborhood of \(F\), that is, \(f\) acts as a holomorphic submersion. This implies that the normal bundle of \(F\) is trivial, and we have \(K_X|_F \simeq K_F\).

In the classical Ohsawa–Takegoshi extension theorem, the metric $ h_L = e^{-\phi_L}$ is assumed to be smooth. However, on Stein or projective manifolds, the theorem can accommodate singular metrics as well, thanks to standard regularization procedures. For more details, the reader is referred to \cite[Section $4$]{DHP13}.




\begin{proof} [The proof of Theorem \ref{schnell-1}]

We only need to demonstrate the existence of one section of $mK_X-f^{\ast}H$ for some $m$, since $mrK_X-f^{\ast}H = r(mK_X-f^{\ast}H) + (r-1)f^{\ast}H$, and $(r-1)f^{\ast}H$ always have many sections for large $r$. Schnell resolved this conjecture under the assumption that $K_Y$ is pseudo-effective (cf. \cite[Theorem 12.1]{Sch22}). His main tools are the existence of Bergman kernel singular Hermitian metrics on pluri-adjoint bundles, and the Koll\'ar-type vanishing theorem for the direct image sheaf, proved by Fujino and Matsumura \cite{FM21}.

Suppose we have a positive integer $m_0$ such that the line bundle $L:= m_0 K_X -f^*H$ is pseudo-effective. This implies that there exists a singular metric $h_L=e^{-\phi_L}$ on $L$ with semi-positive curvature in the sense of current. While the metric may be quite singular, we can find a sufficiently large integer $m_1$ such that 
$$\mathcal{I}(h_L^{\frac{1}{m_1}}) = \mathcal{O}_X \quad \text{and}\quad \mathcal{I}(h_L^{\frac{1}{m_1}}|_{F}) = \mathcal{O}_F.$$

The above equalities can be explained using the concept of Lelong numbers. For any quasi-psh function $\varphi$, a classical theorem of Skoda \cite[Lemma $5.6$]{Dem12a} shows that if $\mu(\varphi, x_0)< 2$, then the multiplier ideal sheaf $\mathcal{I}(\varphi)$ is trivial around $x$, i.e., $e^{-\varphi}$ is integrable in a neighorhood of $x$.

Returning to our case, we can indeed enlarge $m_1$ so that the Lelong number of $\frac{1}{m_1}\phi_L$ is less than $2$. The corresponding multiplier ideal sheaves are trivial by Skoda's theorem. For a point $y =f(F) \in Y$, we can choose an ample line bundle $A$ on $Y$ such that this point $y$ is the common zero set of the sections $(s_i)$ of $A$.
We know $H$ is an ample divisor; therefore, we can multiply it by a large integer $q$ such that 
\begin{enumerate} \label{sss}
    \item $qH- K_Y \geq 2kA$, in the sense that the difference is a semi-positive line bundle,
    \item $qH + lK_Y$ has sections for any non-negative integer $l$; this is the so-called effective non-vanishing result for ample divisors and pseudo-effective divisors.
\end{enumerate}

Now, for any sufficiently large positive integer $r$, we consider the line bundle 
\begin{equation} \label{section1}
(m_0+m_1)r K_X - (r-2q)f^*H = r(m_1K_{X/Y}+L) + f^*qH + f^*(qH+rm_1 K_Y).
\end{equation}
According to the above $(2)$, i.e., the effective non-vanishing result, we only need to show that the first two terms of the right side of \eqref{section1}, the line bundle $r(m_1K_{X/Y}+L) + qf^*H$ is effective. We can write it as
\begin{align*}
    r(m_1K_{X/Y}+L) + qf^*H &=  K_X + (m_1r-1)K_{X/Y} + rL +f^*qH - f^*K_Y  \\
    &=  K_X + \frac{m_1r-1}{m_1} (m_1K_{X/Y} +L) + \frac{1}{m_1} L + f^*qH - f^*K_Y.
\end{align*}
The key point here is that the line bundle $m_1K_{X/Y}+L$ admits a singular Hermitian with a semi-positive curvature current, but its singularity is very mild. This is the main result in \cite[Theorem $1.1$]{PT18}. We denote this metric by $h_{(m_1)}= B_{m_1}^{-1}$. The dual metric $B_{m_1}$ is the so-called twisted $m_1$-th Bergman kernel metric of $-(m_1K_{X/Y}+L)$. Moreover, for any section $u\in H^0(F, m_1 K_F+L)$ and for any positive integer $k$, we have a bounded inequality
$$
|u^k|^2_{h_{(m_1)}^k} \leq O(1)
$$
holds on $F$. Then we define the $\mathbb{Q}$ line bundle 
$$G:= \frac{m_1r-1}{m_1} (m_1K_{X/Y} +L) + \frac{1}{m_1} L + f^*qH - f^*K_Y$$
with the induced singular metric $h_G$ having semi-positive curvature, i.e.,
\begin{equation*}
h_G := h_{(m_1)}^{\frac{m_1r-1}{m_1}}\otimes h_L^{\frac{1}{m_1}} \otimes h_f.
\end{equation*}
Here, $h_f$ is the pullback of the smooth positive metric on $qH-K_Y$. The direct calculation of curvature shows that
$$
\sqrt{-1}\Theta_{h_G}(G) +k\sqrt{-1}\pa\dbar \log|s|^2 \geq f^*\xu \Theta(qH-K_Y-kA).
$$
The right-hand side is semi-positive and more positive than $f^*kA$ due to the choice of $q$. Hence the conditions $(a)$ and $(b)$ in Theorem \ref{ot1} are satisfied. The condition $(c)$ is also easy to achieve by adjusting the smooth metric on the line bundle $A$.

To apply the Ohsawa--Takegoshi extension theorem, we now need to verify the integrability condition. The assumption in the conjecture renders the existence of a section $u_F \in H^0(F, rK_F)$. Let $u_r := u_F^{\otimes(m_0+m_1)} \otimes u_{f^*(qH)} \in H^0(F, r(m_1K_{X/Y}+L) + f^*(qH))=H^0(F, K_X+G)$, where $u_{f^*(qH)}$ is the pullback of some nonzero section of $qH$. We aim to extend $u_r$ from $F$ to the entire $X$ using the above Ohsawa-Takegoshi theorem.
Now we turn to the integrability condition, since
\begin{align*} 
\int_F |u_r|^2_{h_G} dV &= \int_F \big(|u_r^{\otimes m_1}|^2_{h_G^{\otimes m_1}}\big)^{\frac{1}{m_1}} dV \\
&= C \int_F \big(|u_r^{\otimes m_1}|^2_{h_{(m_1)}^{m_1r-1}\cdot h_L} \big)^{\frac{1}{m_1}} dV \\
&\leq  C \int_F \big(|g|^2_{h_L}\big)^{\frac{1}{m_1}} dV \\ &< \infty.
\end{align*}
Here $g$ is a section of $L$ and the final inequality follows from choosing $m_1$ such that $\mathcal{I}(h_L^{\frac{1}{m_1}}|_F) = \mathcal{O}_F$. Therefore, by the Ohsawa--Takegoshi Theorem \ref{ot1}, we can extend $u_r$ to a section $U_r \in H^0(X,  r(m_1K_{X/Y}+L) + qf^*H)$. On the other hand, we know a section $\sigma_r \in H^0(Y, qH+rm_1K_Y)$ exists due to the effective non-vanishing result. Combing this with formula \eqref{section1}, we conclude that $U_r \otimes f^* \sigma_r$ is a section of $(m_0+m_1)r K_X - (r-2q)f^*H$ on $X$. Note that here $r$ needs to be larger than $2q$; thus, the proof is complete.
\end{proof}

\begin{remark}
    According to a theorem of Campana and P\u{a}un: Let \( X \) be a smooth projective variety and let \( f \colon X \to Y \) be a morphism with connected fibers such that \( Y \) is smooth and such that \( K_F \) is pseudoeffective for a general fiber \( F \) of \( f \). Then the divisor
\[
K_{X/Y} - \text{Ram}f
\]
is pseudoeffective. We can relax the assumption on the pseudo-effectiveness of \( K_Y \) to the pseudo-effectiveness of \( K_Y + f(\mathrm{Ram}f) \).
.
\end{remark}

\end{document}